\documentclass[12pt,twoside,english,a4paper]{article}
\usepackage{babel,amssymb,amsthm}
\usepackage{graphicx}
\usepackage{amsmath}
\usepackage{bm}
\usepackage{float}

\newcommand{\punto}{\,\cdot\,}
\newcommand{\ds}{\displaystyle}
\newcommand{\smallfrac}[2]{{\textstyle\frac{#1}{#2}}} 
\newcommand{\jump}[1]{[\![#1]\!]}

\newcommand{\triple}[1]{|\!|\!|#1|\!|\!|}
\newcommand{\bs}{\boldsymbol}

\newtheorem{proposition}{Proposition}[section]
\newtheorem{corollary}[proposition]{Corollary}
\newtheorem{lemma}[proposition]{Lemma}

\numberwithin{equation}{section}

\title{Boundary integral solvers\\ for an evolutionary exterior Stokes problem}

\date{\today}

\author{Constantin Bacuta, Matthew E. Hassell, George C. Hsiao,\\
 \& Francisco--Javier Sayas\footnote{FJS and MH partially funded by NSF grant DMS 1216356.}  \\
Department of Mathematical Sciences, University of Delaware, USA\\
{\tt \{bacuta,mhassell,hsiao,fjsayas\}@udel.edu }}

\begin{document}

\maketitle

\begin{abstract}
This paper proposes and analyzes a full discretization of the exterior transient Stokes problem with Dirichlet boundary conditions. The method is based on a single layer boundary integral representation, using Galerkin semidiscretization in the space variables and multistep Convolution Quadrature in time. Convergence estimates are based on a Laplace domain analysis, which translates into a detailed study of the exterior Brinkman problem. Some numerical experiments are provided.\\
{\bf AMS Subject classification.} 65R20, 65M38\\
{\bf Keywords.} exterior Stokes problem, convolution quadrature, boundary element method
\end{abstract}

\section{Introduction}

In this paper we propose a fully discrete method based on an integral equation for the exterior Stokes problem with Dirichlet boundary conditions in two or three dimensions. The integral equation is based on a single layer potential representation of the velocity field. The numerical discretization uses a general Galerkin scheme for semidiscretization in space and Convolution Quadrature \cite{Lubich:1988} for discretization in time. The analysis is carried out by combining ideas of Laplace domain analysis of integral operators \cite{LuSc:1992} with the transformation of the Galerkin-BEM discretization in space into a set of exotic transmission conditions \cite{LaSa:2009}. As part of the paper, we include a novel analysis of the single layer potential and operator for the Stokes resolvent equations (the Brinkman equations) on a general Lipschitz domain. 

The literature on numerical methods for integral representations of parabolic problems has focused extensively on the heat equation. Most theoretical results are based on the single-layer representation, leading to a Volterr\`a-Fredholm integral equation that can be formally considered to be of the first kind. (We note that the mapping properties of the integral operators make the integral equations of the second kind for parabolic problems not to be a smooth perturbation of the identity, due to the mapping properties in the time variable.  Additional complications arise when the boundary is not smooth.) This analysis was sparked by the work of Arnold and Noon \cite{ArNo:1989} and Costabel \cite{Costabel:1990}, with some sequels as \cite{HsSa:1991}. The work of Lubich and Schneider \cite{LuSc:1992} offered a numerical treatment of the heat equation single layer operator equation. Other formulations, including fast multiplication techniques, appear in recent work of Tausch \cite{Tausch:2007, Tausch:2009, MeScTa:2014}. The mathematical literature for the unsteady exterior Stokes problem using integral equations seems to be quite limited: see, for instance, \cite{HeHs:1993}, \cite{HeHs:2007}. A general overview of the state of the art of time domain integral equations one decade ago can be found in \cite{Costabel:2004}.

For our analysis we will rely on properties of the Brinkman single layer potential. We will however take a different approach than the one given in \cite{KoPo:2004, Kohr:2007, KoWe:2009}, since we need to study the behavior of all the bounds as functions of the parameter in the Brinkman model. We will adopt a Laplace domain approach similar to the one used in \cite{BaHa:1986} for the wave equation. For some technical issues, we will rely on recent results on the Stokes potentials on general Lipschitz domains \cite{SaSe:2014}. The passage to the time-domain will be done with a modification of a result in \cite{LuSc:1992}. Following \cite{LaSa:2009} we will analyze the semidiscretization in space in a systematic way, showing that a postprocessed solution (the velocity field) can have better properties than the preprocessed solution (the boundary density and, therefore, the pressure field, which is postprocessed with a steady-state operator). Finally, we will apply a general multistep-based Convolution Quadrature strategy and analyze it using the results in \cite{Lubich:1988}. We note that this final step will be the only one where we will not be able to analyze how the constants that appear in the error estimates depend on time (as the latter grows to infinity).

The paper starts with two long sections (Sections \ref{sec:2} and \ref{sec:3}) presenting the integral and variational forms of the single-layer potential for the Brinkman problem and providing bounds in terms of the parameter of the Brinkman equation. In Section \ref{sec:4} we transfer the Laplace domain estimates to estimates for the transient single layer potential for the Stokes equation, using a technical result that is proved in Appendix \ref{app:A}. In Section \ref{sec:5} we introduce and analyze a general Galerkin semidiscretization in space of the integral equation. We provide bounds for the semidiscretization in space (Galerkin error operator) plus some stability bounds (Galerkin solver) that are needed for the analysis of the fully discrete method. In Section \ref{sec:6} we present and analyze the fully discrete scheme and show some numerical experiments. Finally, Appendix \ref{sec:AppB} shows an alternative integral formulation that can be used to eliminate some inconvenient Lagrange multipliers that are needed to impose conformity restrictions in the boundary element space.

\paragraph{Foreword on background and notation.} We will use basic properties of Sobolev spaces on bounded Lipschitz domains and on their boundaries \cite{AdFo:2003}. All aspects related to integral operators will be proved using variational techniques \cite{SaSe:2014}. The passage to the time-domain requires the momentary use of basic vector-valued distribution theory. It is important to remark that {\em all brackets will be taken to be bilinear}, even if they are employed in the context of complex-valued functions. In particular, for scalar fields complex-valued $u,v$, vector fields $\mathbf u,\mathbf v$ and matrix-valued fields (tensors) $\mathrm U,\mathrm V$, and an open set $\mathcal O\subset\mathbb R^d$, we will denote
\[
(u,v)_{\mathcal O}:=\int_{\mathcal O} u\,v
\qquad
(\mathbf u,\mathbf v)_{\mathcal O}:=\int_{\mathcal O} \mathbf u\cdot\mathbf v
\qquad
(\mathrm U,\mathrm V)_{\mathcal O}:=\int_{\mathcal O} \mathrm U:\mathrm V
\]
where $\mathrm U:\mathrm V:=\sum_{i,j} \mathrm U_{ij}\mathrm V_{ij}$. Given a Hilbert space $X$, we will write $\mathbf X:=X^d$ and immediately assume it to be endowed with the product topology. 

\section{The Brinkman single layer potential}\label{sec:2}

In this section we present the variational theory for the Brinkman single layer potential as a holomorphic function of its parameter. This is equivalent to studying the single layer potential associated to the resolvent Stokes problem. In all the following arguments, the parameter $s$ is a complex number not in the negative real axis
\[
s\in \mathbb C_\star:=\mathbb C\setminus (-\infty,0].
\]
The space of solenoidal vector fields
\begin{equation}\label{eq:N2.1}
\widehat{\mathbf V}(\mathbb R^d):=\{\mathbf u\in \mathbf H^1(\mathbb R^d)\,:\, \mathrm{div}\,\mathbf u=0\}
\end{equation}
will also play a key role. The geometric setting is as follows: we consider a bounded Lipschitz domain $\Omega_-$, with connected boundary $\Gamma$, and the associated unbounded exterior domain $\Omega_+:=\mathbb R^d\setminus\overline{\Omega_-}$. The superindices $\pm$ will be used to refer to limits/traces from $\Omega_\pm$. We will use the angled bracket $\langle\,\cdot\,,\,\cdot\,\rangle_\Gamma$ to denote the $L^2(\Gamma)$ and $\mathbf L^2(\Gamma)$ inner products (with the above convention on not conjugating any component) as well as its extension to duality products between the spaces $H^{\pm1/2}(\Gamma)$, as well as between their vector-valued counterparts.

\paragraph{Jumps of traces and normal stresses}
The jump of the trace across $\Gamma$, for a locally $\mathbf H^1$ function, is defined as
$
\jump{\gamma\mathbf v}:=\gamma^-\mathbf v-\gamma\mathbf v^+.
$
Let now $\mathbf u\in \mathbf H^1(\mathbb R^d\setminus\Gamma)$ and $p\in L^2(\mathbb R^d\setminus\Gamma$), be such that
\[
\mathbf f:=-2\nu\mathrm{div}\,\bs\varepsilon(\mathbf u)+\nabla p \in L^2(\mathbb R^d\setminus\Gamma), 
\qquad \bs\varepsilon(\mathbf u):=\smallfrac12(\mathrm D\mathbf u+(\mathrm D\mathbf u)^\top).
\]
We can then define the functionals $\mathbf t^\pm(\mathbf u,p)\in \mathbf H^{-1/2}(\Gamma)$ given by the relations:
\begin{eqnarray*}
\langle \mathbf t^-(\mathbf u,p),\gamma\mathbf v\rangle_\Gamma &:=& 2\nu \left( \boldsymbol\varepsilon(\mathbf u),\boldsymbol\varepsilon(\mathbf v)\right)_{\Omega_-}-(p,\mathrm{div}\,\mathbf v)_{\Omega_-}
-(\mathbf f,\mathbf v)_{\Omega_-}\qquad \forall \mathbf v\in \mathbf H^1(\Omega_-),\\
\langle \mathbf t^+(\mathbf u,p),\gamma\mathbf v\rangle_\Gamma &:=& -2\nu \left( \boldsymbol\varepsilon(\mathbf u),\boldsymbol\varepsilon(\mathbf v)\right)_{\Omega_+}+(p,\mathrm{div}\,\mathbf v)_{\Omega_+}
+(\mathbf f,\mathbf v)_{\Omega_+}
\qquad \forall \mathbf v\in \mathbf H^1(\Omega_+).
\end{eqnarray*}
We can thus define the jump of the normal stress
$
\jump{\mathbf t(\mathbf u,p)}:=\mathbf t^-(\mathbf u,p)-\mathbf t^+(\mathbf u,p).
$
In spite of the global definitions of $\mathbf t^\pm(\mathbf u,p)$ using test functions in $\mathbf H^1(\Omega_\pm)$, it is clear that these operators have a local behavior and can be extended to pairs $(\mathbf u,p)$ that only exhibit the required properties in a neighborhood of the boundary. This subtle distinction will be employed in the two dimensional case, where the pressure $p$ is only locally in $L^2$. In particular we will use the formula
\begin{equation}\label{eq:N2.2}
\langle \jump{\mathbf t(\mathbf u,p)},\gamma\mathbf v\rangle_\Gamma=a(\mathbf u,\mathbf v)-(p,\mathrm{div}\,\mathbf v)_{\mathbb R^d}+(-2\nu\mathrm{div}\,\bs\varepsilon(\mathbf u)+\nabla p,\mathbf v)_{\mathbb R^d\setminus\Gamma}
\qquad\forall \mathbf v\in \mathcal D(\mathbb R^d)^d,
\end{equation}
where
\[
a(\mathbf u,\mathbf v):=2\nu \left( \boldsymbol\varepsilon(\mathbf u),\boldsymbol\varepsilon(\mathbf v)\right)_{\mathbb R^d}
\]
and $\mathcal D(\mathbb R^d)$ is the set of $\mathcal C^\infty$ compactly supported functions. 

\subsection{Integral forms}

The following definitions can be found in \cite{Kohr:2007}, \cite{KoWe:2009}, \cite[p.81]{KoPo:2004}

\paragraph{The pressure  potential.} For a given density $\boldsymbol\lambda\in \mathbf H^{-1/2}(\Gamma)$, we define
\[
(\mathrm S_p\boldsymbol\lambda)(\mathbf z):=\langle 
	\mathbf e_p(\mathbf z-\cdot),\boldsymbol\lambda \rangle_\Gamma, \qquad \mathbf z\in \mathbb R^d\setminus\Gamma,
\]
where
\[
\mathbf e_p(\mathbf r):=\frac1{2(d-1)\pi}\frac1{r^d}\mathbf r
\]	
is the negative gradient of the fundamental solution to the Laplace equation. The behavior at infinity of $\mathbf e_p$ gives different properties for the operator $\mathrm S_p$ in two and three dimensions. In the two dimensional case, the
closed subspace
\[
\mathbf H^{-1/2}_0(\Gamma):=\{\boldsymbol\lambda \in \mathbf H^{-1/2}(\Gamma)\,:\, \langle\boldsymbol\lambda,\mathbf a \rangle_\Gamma=0\quad\forall \mathbf a \in \boldsymbol P_0(\Gamma)\}=\boldsymbol P_0(\Gamma)^\circ
\]
plays an important role. As a simple fact that this pressure part of the single layer potential is the same for the Brinkman as for the Stokes problems we can show the following result \cite[Propositions 5.2 and 7.2]{SaSe:2014}:
\begin{proposition}\label{prop:2.1}
\
\begin{itemize}
\item[{\rm (a)}] When $d=3$, $\mathrm S_p:\mathbf H^{-1/2}(\Gamma) \to L^2(\mathbb R^3)$ is bounded.
\item[{\rm (b)}] When $d=2$, $\mathrm S_p:\mathbf H^{-1/2}_0(\Gamma) \to L^2(\mathbb R^2)$ is bounded.
\end{itemize}
\end{proposition}

\paragraph{First order asymptotics in the two dimensional case.}
We note that the zero integral  condition in the definition of the space $\mathbf H^{-1/2}_0(\Gamma)$ only affects the behavior at infinity of $\mathrm S_p\boldsymbol\lambda$. We therefore explore the first order asymptotics at infinity of $\mathrm S_p\boldsymbol\lambda$ for general $\boldsymbol\lambda$.
Expanding the kernel function $\mathbf e_p$, we can write
\[
(\mathrm S_p\boldsymbol\lambda)(\mathbf z)=\frac1{2\pi}\frac1{1+|\mathbf z|^2}\langle \mathbf z , \bs\lambda\rangle_\Gamma + \mathcal O(|\mathbf z|^{-2}), \qquad \mbox{as $|\mathbf z|\to \infty$.}
\]
Therefore, apart from the leading term, the pressure potential is in $L^2(\mathbb R^2)$. Let then
\begin{equation}\label{eq:2.9}
\mathbf p_\infty(\mathbf x):= \frac1{2\pi}\frac1{1+|\mathbf x|^2}\,\mathbf x, \qquad (\mathrm D \mathbf p_\infty)(\mathbf x)=\frac1{2\pi}\frac1{1+|\mathbf x|^2} \Big( \mathrm I -\frac2{1+|\mathbf x|^2} \mathbf x\otimes\mathbf x\Big),
\end{equation}
\begin{equation}\label{eq:2.10}
\mathbf g_\ell:=(\mathrm D\mathbf p_\infty)\mathbf e_\ell=\nabla (\mathbf p_\infty\cdot\mathbf e_\ell), \qquad \jmath_\ell(\boldsymbol\lambda):=\langle\boldsymbol\lambda,\mathbf e_\ell\rangle_\Gamma \qquad \ell\in \{1,2\}.
\end{equation}
Here $\{\mathbf e_1,\mathbf e_2\}$ is the canonical basis for $\mathbb R^2.$
This leads to the proof of the following result.

\begin{proposition}\label{prop:2.2}
When $d=2$,
$
\mathrm S_p-\sum_{\ell=1}^2 (\mathbf p_\infty\cdot\mathbf e_\ell)\jmath_\ell:\mathbf H^{-1/2}(\Gamma) \to \mathbf L^2(\mathbb R^2)
$
is bounded.
\end{proposition}

We note that while the result  in Proposition \ref{prop:2.1} is a direct consequence of what is known for the Stokes operator (see \cite{SaSe:2014} for a fully developed variational and integral theory), the decomposition of Proposition \ref{prop:2.2} (that subtracts an easily identifiable first order term from the potential) seems to be new.

\paragraph{The velocity potential.} For $\boldsymbol\lambda\in \mathbf H^{-1/2}(\Gamma)$, we define
\begin{equation}\label{eq:N2.5}
(\mathrm S_u(s)\boldsymbol\lambda)(\mathbf z)
:=\left\langle\mathrm E_u(\mathbf z-\,\cdot\,;s),\boldsymbol\lambda\right\rangle_\Gamma,
\end{equation}
where
\begin{equation}\label{eq:2.1}
\mathrm E_u(\mathbf r;s):=
	\frac1{4(d-1)\pi\nu} \left( 
		\frac{A_d(\sqrt{s}\, r)}{r^{d-2}}\,\mathrm I+ \frac{B_d(\sqrt{s}\, r)}{r^d}\,\mathbf r\otimes\mathbf r\right),
	\qquad \mathbf z\in \mathbb R^d\setminus\Gamma,
\end{equation}
and
\begin{eqnarray*}
A_3(z)&:=&2e^{-z}(1+z^{-1}+z^{-2})-2z^{-2}=2z^{-2} \left(e^{-z}(z^2+z+1)- 1\right),\\
B_3(z)&:=&-2e^{-z} (1+3z^{-1}+3z^{-2})+ 6 z^{-2}=-2z^{-2} \left( e^{-z}(z^2+3z+3)-3\right),\\
A_2(z) &:=& 2 (K_0(z)+z^{-1} K_1(z)-z^{-2}),\\
B_2(z) &:=& 2(-K_0(z)- 2 z^{-1} K_1(z)+ 2z^{-2})=2(2z^{-2}-K_2(z)),
\end{eqnarray*}
$K_\ell$ being the modified Bessel function of order $\ell$. The square root in \eqref{eq:2.1} is the one determination of the square root that is analytic in $\mathbb C_\star$. Note that
 $A_3$ and $B_3$ are entire functions with $A_3(0)=B_3(0)=1$. The functions $A_2$ and $B_2$ are only analytic in $\mathbb C_\star$, and have logarithmic singularities in the cut $(-\infty,0]$. Comparing the integral expressions of the Brinkman potential with those of the Stokes potential, it is possible to prove that for any $s\in \mathbb C_\star$
the operator $\mathrm S_u(s): \mathbf H^{-1/2}(\Gamma)\to \mathbf H^1(\mathbb R^d)$ is bounded. We are however interested in the dependence on $s$ of the bounds for this potential and some related integral operators.

\subsection{Variational theory in three dimensions}

\begin{proposition}[Existence via potential theory]\label{prop:4.1}
Let $\boldsymbol\lambda\in \mathbf H^{-1/2}(\Gamma)$ and consider the functions
$
\mathbf u_\lambda:=\mathrm S_u(s)\boldsymbol\lambda \in \mathbf H^1(\mathbb R^3)
$
and
$p_\lambda:=\mathrm S_p\boldsymbol\lambda\in L^2(\mathbb R^3).
$
Then
\begin{subequations}\label{eq:4.1}
\begin{alignat}{4}
\label{eq:4.1a}
 -2\nu \mathrm{div}\,\boldsymbol\varepsilon(\mathbf u_\lambda)+s\mathbf u_\lambda+\nabla p_\lambda = \mathbf 0 & \qquad & \mbox{in $\mathbb R^3\setminus\Gamma$},\\
\label{eq:4.1b}
 \mathrm{div}\,\mathbf u_\lambda = 0 & & \mbox{in $\mathbb R^3\setminus\Gamma$},\\
\label{eq:4.1c}
 \jump{\gamma\mathbf u_\lambda}=0, & & \\
\label{eq:4.1d}
 \jump{\mathbf t(\mathbf u_\lambda,p_\lambda)}= \boldsymbol\lambda. & &
\end{alignat}
\end{subequations}
Moreover, a pair $(\mathbf u_\lambda,p_\lambda)\in \mathbf H^1(\mathbb R^3)\times L^2(\mathbb R^3)$ is a solution of \eqref{eq:4.1} if and only if
\begin{equation}\label{eq:4.2}
\left[ \begin{array}{l}
\mathbf u_\lambda\in \mathbf H^1(\mathbb R^3), p_\lambda\in L^2(\mathbb R^3),\\[1.5ex]
\begin{array}{rll} \ds a(\mathbf u_\lambda,\mathbf v)+s(\mathbf u_\lambda,\mathbf v)_{\mathbb R^3}-(p_\lambda,\mathrm{div}\,\mathbf v)_{\mathbb R^3} &=\langle\boldsymbol\lambda,\gamma\mathbf v\rangle_\Gamma & \forall \mathbf v\in \mathbf H^1(\mathbb R^3),\\[1.5ex]
(\mathrm{div}\,\mathbf u_\lambda,q)_{\mathbb R^3} &=0 & \forall q \in L^2(\mathbb R^3).
\end{array}
\end{array} \right.
\end{equation}
\end{proposition}

\begin{proof} Regularity is guaranteed by the properties of the integral formulations of the layer potentials. The differential equations \eqref{eq:4.1a}-\eqref{eq:4.1b} are satisfied pointwise in a strong sense (this can be proved by differentiation directly in the fundamental solutions), and therefore, they are satisfied in a distributional sense. Condition \eqref{eq:4.1c} is a direct consequence of the fact that $\mathbf u_\lambda\in \mathbf H^1(\mathbb R^3).$ Finally, condition \eqref{eq:4.1d} follows from \eqref{eq:4.1a} and \eqref{eq:N2.2}. The equivalence of \eqref{eq:4.1} and \eqref{eq:4.2} is straightforward.
\end{proof}

\begin{proposition}[Variational form in solenoidal spaces]\label{prop:4.3}
If  $(\mathbf u_\lambda,p_\lambda)$ is a solution of \eqref{eq:4.2}, then $\mathbf u_\lambda$ is a solution of
\begin{equation}\label{eq:4.3}
\left[\begin{array}{l}\mathbf u_\lambda\in \widehat{\mathbf V}(\mathbb R^3),\\[1.5ex]
 a(\mathbf u_\lambda,\mathbf v)+s(\mathbf u_\lambda,\mathbf v)_{\mathbb R^3}=\langle\boldsymbol\lambda,\gamma\mathbf v\rangle_\Gamma\qquad \forall \mathbf v\in \widehat{\mathbf V}(\mathbb R^3).
\end{array}\right.
\end{equation}
Problem \eqref{eq:4.3} is well posed.
\end{proposition}

\begin{proof} Taking $\mathbf v\in \widehat{\mathbf V}(\mathbb R^3)$ as test functions in \eqref{eq:4.2} it is clear that $\mathbf u_\lambda$ satisfies the equations \eqref{eq:4.3}. To prove well posedness, note that
\[
\|\mathbf u\|_{1,\mathbb R^3}^2\le 2 \|\boldsymbol\varepsilon(\mathbf u)\|_{\mathbb R^3}^2+\|\mathbf u\|_{\mathbb R^3}^2\le 2\|\mathbf u\|_{1,\mathbb R^3}^2 \qquad \forall \mathbf u\in \boldsymbol{\mathcal D}(\mathbb R^3). 
\]
Applying a density argument, the bilinear form of \eqref{eq:4.3} is shown to be coercive.
\end{proof}

\begin{corollary}
Problem  \eqref{eq:4.2} has a unique solution. Problem \eqref{eq:4.1} has a unique solution in $\mathbf H^1(\mathbb R^3)\times L^2(\mathbb R^3)$.
\end{corollary}

\begin{proof} {\em Uniqueness.}
If $(\mathbf u,p)$ is a solution of the corresponding homogeneous problem, by Proposition \ref{prop:4.3} it follows that $\mathbf u\equiv\mathbf 0$. Therefore $p\in L^2(\mathbb R^3)$ satisfies $\nabla p=\mathbf 0$, which implies that $p\equiv 0$. {\em Existence.} Proposition \ref{prop:4.1} shows existence of solution using the integral form of the potentials.
\end{proof}

\begin{corollary}
$\mathrm S_u(s)\mathbf n\equiv \mathbf 0$
\end{corollary}

\begin{proof}
It is clear that $(\mathbf 0,-\chi_{\Omega_-})$ is a solution of \eqref{eq:4.2} with $\boldsymbol\lambda=\mathbf n$. By uniqueness, this is the layer potential.
\end{proof}

\subsection{Variational theory in two dimensions}

At this stage, the main difference between the two and three dimensional cases arises from the fact that $p=\mathrm S_p \boldsymbol\lambda\not\in L^2(\mathbb R^2)$ if $\boldsymbol\lambda\not\in \mathbf H^{-1/2}_0(\Gamma)$. Note that the condition $\boldsymbol\lambda\in\mathbf H^{-1/2}_0(\Gamma)$ is natural in the two-dimensional Stokes equation \cite[Proposition 3.2]{SaSe:2014} and it is somehow due to the fact that constant functions are elements of the associated weighted Sobolev spaces. This is not the case for the Brinkman problem. The following approach  uses the precise knowledge of the asymptotic behavior of the single layer potential for the pressure variable.

\begin{proposition}[Existence via potential theory]
If $\boldsymbol\lambda\in \mathbf H^{-1/2}(\Gamma)$ and we consider the functions
$
\mathbf u_\lambda:=\mathrm S_u(s)\boldsymbol\lambda \in \mathbf H^1(\mathbb R^2), 
$
and
$p_\lambda:=\mathrm S_p\boldsymbol\lambda\in L^2_{\mathrm{loc}}(\mathbb R^2),
$
then
\begin{subequations}\label{eq:4.4}
\begin{alignat}{4}
 -2\nu \mathrm{div}\,\boldsymbol\varepsilon(\mathbf u_\lambda)+s\mathbf u_\lambda+\nabla p_\lambda = \mathbf 0 & \qquad & \mbox{in $\mathbb R^2\setminus\Gamma$},\\
 \mathrm{div}\,\mathbf u_\lambda = 0 & & \mbox{in $\mathbb R^2\setminus\Gamma$},\\
 \jump{\gamma\mathbf u_\lambda}=0, & & \\
 \jump{\mathbf t(\mathbf u_\lambda,p_\lambda)}= \boldsymbol\lambda, & &
\end{alignat}
\end{subequations}
\end{proposition}

\begin{proof}
This is just a consequence of the properties of the associated integral operators.
\end{proof}

\begin{proposition}[Variational form in solenoidal spaces]\label{prop:4.8}
Let $\boldsymbol\lambda\in \mathbf H^{-1/2}(\Gamma)$, and $\mathbf u_\lambda:=\mathrm S_u(s)\boldsymbol\lambda$. Then $\mathbf u_\lambda$ is the unique solution of the variational problem 
\begin{equation}\label{eq:4.5}
\left[\begin{array}{l}\mathbf u_\lambda \in \widehat{\mathbf V}(\mathbb R^2),\\[1.5ex]
 a(\mathbf u_\lambda,\mathbf v)+s(\mathbf u_\lambda,\mathbf v)_{\mathbb R^2}=\langle\boldsymbol\lambda,\gamma\mathbf v\rangle_\Gamma\qquad \forall \mathbf v\in \widehat{\mathbf V}(\mathbb R^2).
\end{array}\right.
\end{equation}
\end{proposition}

\begin{proof} Consider $\mathbf p_\infty$, $\mathbf g_\ell$, and $\jmath_\ell$ as defined in \eqref{eq:2.9} and \eqref{eq:2.10}.
Note that we can write (see Proposition \ref{prop:2.2})
\[
p_\lambda=\sum_{\ell=1}^2  \jmath_\ell(\boldsymbol\lambda)\,\mathbf p_\infty\cdot\mathbf e_\ell+p_{\mathrm{reg},\lambda}, \quad p_{\mathrm{reg},\lambda}\in L^2(\mathbb R^2),
\]
and
\[
\nabla p_\lambda = \sum_{\ell=1}^2 \jmath_\ell(\boldsymbol\lambda)\mathbf g_\ell +\nabla p_{\mathrm{reg},\lambda}, \qquad \mathbf g_\ell\in \mathbf L^2(\mathbb R^2).
\]
Also
\[
\boldsymbol\lambda=\jump{\mathbf t(\mathbf u_\lambda,p_\lambda)}=\jump{\mathbf t(\mathbf u_\lambda,p_{\mathrm{reg},\lambda})}+\sum_{\ell=1}^2\jmath_\ell(\boldsymbol\lambda)\jump{\mathbf t(\mathbf 0,\mathbf p_\infty\cdot\mathbf e_\ell)}=\jump{\mathbf t(\mathbf u_\lambda,p_{\mathrm{reg},\lambda})}.
\]
Therefore $(\mathbf u_\lambda,p_{\mathrm{reg},\lambda})\in \mathbf H^1(\mathbb R^2)\times L^2(\mathbb R^2)$ is a solution of
\begin{subequations}\label{eq:4.6}
\begin{alignat}{4}
 -2\nu \mathrm{div}\,\boldsymbol\varepsilon(\mathbf u_\lambda)+s\mathbf u_\lambda+\nabla p_{\mathrm{reg},\lambda} = -\sum_{\ell=1}^2 \jmath_\ell(\boldsymbol\lambda)\mathbf g_\ell & \qquad & \mbox{in $\mathbb R^2\setminus\Gamma$},\\
 \mathrm{div}\,\mathbf u_\lambda = 0 & & \mbox{in $\mathbb R^2\setminus\Gamma$},\\
 \jump{\gamma\mathbf u_\lambda}=0, & & \\
 \jump{\mathbf t(\mathbf u_\lambda,p_{\mathrm{reg},\lambda})}= \boldsymbol\lambda, & &
\end{alignat}
\end{subequations}
but this problem is equivalent to
\begin{equation}\label{eq:4.7}
\left[ \begin{array}{l}
\mathbf u_\lambda\in \mathbf H^1(\mathbb R^2), p_{\mathrm{reg},\lambda}\in L^2(\mathbb R^2),\\[1.5ex]
\begin{array}{ll} \ds a(\mathbf u_\lambda,\mathbf v)+s(\mathbf u_\lambda,\mathbf v)_{\mathbb R^2}-(p_{\mathrm{reg},\lambda},\mathrm{div}\,\mathbf v)_{\mathbb R^2} \\
\hspace{3cm} =\langle\boldsymbol\lambda,\gamma\mathbf v\rangle_\Gamma-\sum_{\ell=1}^2\jmath_\ell(\boldsymbol\lambda)\, (\mathbf g_\ell,\mathbf v)_{\mathbb R^2} & \forall \mathbf v\in \mathbf H^1(\mathbb R^2),\\[1.5ex]
(\mathrm{div}\,\mathbf u_\lambda,q)_{\mathbb R^2} =0  & \forall q \in L^2(\mathbb R^2).
\end{array}
\end{array} \right.
\end{equation}
Testing with $\mathbf v\in \widehat{\mathbf V}(\mathbb R^2)$ we obtain the problem
\begin{equation}\label{eq:4.8}
\left[\begin{array}{l}\mathbf u_\lambda \in \widehat{\mathbf V}(\mathbb R^2),\\[1.5ex]
 a(\mathbf u_\lambda,\mathbf v)+s(\mathbf u_\lambda,\mathbf v)_{\mathbb R^2}=\langle\boldsymbol\lambda,\gamma\mathbf v\rangle_\Gamma-\sum_{\ell=1}^2 \jmath_\ell(\boldsymbol\lambda) (\mathbf g_\ell,\mathbf v)_{\mathbb R^2}\qquad \forall \mathbf v\in \widehat{\mathbf V}(\mathbb R^2).
\end{array}\right.
\end{equation}
We next notice that
\[
\mathbf p_\infty\cdot\mathbf e_\ell \in W(\mathbb R^2):=\{ u:\mathbb R^2\to\mathbb R\,:\, \rho u\in L^2(\mathbb R^2),
\quad\nabla u\in \mathbf L^2(\mathbb R^2)\}, 
\]
where (see \cite{AmGiGi:1992})
\[
\rho(\mathbf x):=\frac1{1+\smallfrac12\log(1+|\mathbf x|^2)}
\frac1{\sqrt{1+|\mathbf x|^2}}. 
\]
By density of smooth compactly supported functions in $W(\mathbb R^2)$ \cite{AmGiGi:1992}, it follows that there exists $\{\varphi_n\} \subset \mathcal D(\mathbb R^2)$ such that $\nabla \varphi_n \to \mathbf g_\ell=\nabla (\mathbf p_\infty\cdot\mathbf e_\ell)$. Therefore
\[
(\mathbf g_\ell,\mathbf v)_{\mathbb R^2}=\lim_{n\to \infty} (\nabla \varphi_n,\mathbf v)_{\mathbb R^2}=-\lim_{n\to \infty} (\varphi_n,\mathrm{div}\,\mathbf v)_{\mathbb R^2}=0 \qquad \forall \mathbf v\in \widehat{\mathbf V}(\mathbb R^2),
\] 
which shows that problems \eqref{eq:4.5} and \eqref{eq:4.8} are the same.
\end{proof}

\section{Bounds in the Laplace domain}\label{sec:3}

In this section we study properties of the operators $\mathrm S(s):=\mathrm S_u(s)$ (see \eqref{eq:N2.5}), $\mathrm V(s):=\gamma \mathrm S(s)$, $\mathrm V(s)^{-1}$ and $\mathrm S(s)\mathrm V(s)^{-1}(s)$ as functions of $s\in \mathbb C_\star.$ We start with two technical results.

\begin{lemma}\label{lemma:4.1}
Let
\[
\mathbf H^{1/2}_n(\Gamma):=\{\boldsymbol\xi\in \mathbf H^{1/2}(\Gamma)\,:\, \int_\Gamma \boldsymbol\xi\cdot\mathbf n=0\}
\]
and $\widehat{\mathbf V}(\mathbb R^d)$ be as defined in \eqref{eq:N2.1}
Then the trace operator $\gamma :\widehat{\mathbf V}(\mathbb R^d) \to \mathbf H^{1/2}_n(\Gamma)$ is surjective.
\end{lemma}

\begin{proof} In \cite[Proposition 4.4]{SaSe:2014} there is a right inverse whose range contains only compactly supported functions. The same right inverse is valid now.
\end{proof}

\begin{lemma}{\rm \cite[Propositions 4.1 and 4.2]{SaSe:2014}}\label{lemma:4.2}
Let 
\[
\mathbf H^{-1/2}_m(\Gamma):=\{ \boldsymbol\lambda\in \mathbf H^{-1/2}(\Gamma)\,:\, \langle \boldsymbol\lambda,\mathbf m\rangle_\Gamma=0\}, \qquad \mathbf m(\mathbf x):=\mathbf x.
\]
Then the decompositions
\[
\mathbf H^{1/2}(\Gamma)=\mathbf H^{1/2}_n(\Gamma)\oplus \mathrm{span}\{ \mathbf m\}\qquad \mbox{and}\qquad \mathbf H^{-1/2}(\Gamma)=\mathbf H^{-1/2}_m (\Gamma) \oplus \mathrm{span}\,\{\mathbf n\}
\]
are stable and there exists $C>0$ such that
\[
\|\boldsymbol\lambda\|_{-1/2,\Gamma}\le C \sup_{\mathbf 0\neq\boldsymbol\xi\in \mathbf H^{1/2}_n(\Gamma)}\frac{|\langle\boldsymbol\lambda,\boldsymbol\xi\rangle_\Gamma|}{\|\boldsymbol\xi\|_{1/2,\Gamma}}\qquad \forall\boldsymbol\lambda \in \mathbf H^{-1/2}_m(\Gamma).
\]
\end{lemma}

\paragraph{Some technicalities}
Given $s\in \mathbb C_\star=\mathbb C\setminus(-\infty,0]$, we take its square root $s^{1/2}:=|s|^{1/2} \exp(\smallfrac\imath2\,\mathrm{Arg}\,s) \in \mathbb C_+$ and denote
\begin{equation}\label{eq:5.0}
\omega:=\mathrm{Re}\,s^{1/2}=\mathrm{Re}\,\overline s^{1/2}. \qquad \underline\omega:=\min\{1,\omega\}=\min\{1,\mathrm{Re}\,s^{1/2}\}.
\end{equation}
We also consider the norms (depending on $|s|$)
\[
\triple{\mathbf u}_{(s)}^2:=2\nu \|\boldsymbol\varepsilon(\mathbf u)\|_{\mathbb R^d}^2+|s| \,\|\mathbf u\|_{\mathbb R^d}^2.
\]
Note that
\begin{equation}\label{eq:5.1}
\alpha_1(s) \triple{\mathbf u}_{(1)}\le \triple{\mathbf u}_{(s)}\le \alpha_2(s) \triple{\mathbf u}_{(1)}\qquad \forall \mathbf u\in \mathbf H^1(\mathbb R^d),\quad\forall s\in \mathbb C_\star,
\end{equation}
where
\begin{equation}\label{eq:5.2}
\alpha_1(s):=\min\{1,|s|^{1/2}\}\ge \underline\omega \qquad \alpha_2(s):=\max\{1,|s|^{1/2}\} \le \frac{|s|^{1/2}}{\underline\omega}\qquad \forall s\in \mathbb C_\star.
\end{equation}
The norm $\triple{\cdot}_{(1)}$ will be used as the standard norm in $\mathbf H^1(\mathbb R^d)$. Note finally that
\begin{equation}\label{eq:5.3}
| a(\mathbf u,\mathbf v)+s\,(\mathbf u,\mathbf v)_{\mathbb R^d}|\le \triple{\mathbf u}_{(s)}\triple{\mathbf v}_{(s)}.
\end{equation}

\begin{proposition}[Properties of the single layer operator]\label{prop:N3.3}
\
\begin{itemize}
\item[{\rm (a)}] (Symmetry) 
\[
\langle \boldsymbol\lambda,\mathrm V(s)\boldsymbol\mu\rangle_\Gamma = \langle \boldsymbol\mu,\mathrm V(s)\boldsymbol\lambda\rangle_\Gamma \qquad \forall \boldsymbol\lambda,\boldsymbol\mu\in \mathbf H^{-1/2}(\Gamma).
\]
\item[{\rm (b)}] (Positivity) 
\[
\mathrm{Re}\, \langle \overline s^{1/2}\overline{\boldsymbol\lambda},\mathrm V(s)\boldsymbol\lambda\rangle_\Gamma = \omega \triple{\mathrm S(s)\bs\lambda}_{(s)}^2\qquad \forall \boldsymbol\lambda\in \mathbf H^{-1/2}(\Gamma),
\]
\item[{\rm (c)}] $\mathrm{Ker}\,\mathrm V(s)=\mathrm{span}\,\{\mathbf n\}$
\item[{\rm (d)}] (Coercivity) There exists $C>0$ such that
\[
|\langle \overline{\boldsymbol\lambda},\mathrm V(s)\boldsymbol\lambda\rangle_\Gamma| \ge C \, \frac{\omega}{|s|^{1/2} \max\{1,|s|\}}\|\boldsymbol\lambda\|_{-1/2,\Gamma}^2 \ge C \frac{\omega\underline\omega^2}{|s|^{3/2}}\|\boldsymbol\lambda\|_{-1/2,\Gamma}^2\qquad\forall\boldsymbol\lambda\in \mathbf H^{-1/2}_m(\Gamma).
\]
Therefore $\mathrm V(s):\mathbf H^{-1/2}_m(\Gamma) \to \mathbf H^{1/2}_n(\Gamma)$ is invertible.
\end{itemize}
\end{proposition}

\begin{proof}
Let $\mathbf u_\lambda:=\mathrm S(s)\boldsymbol\lambda$ and $\mathbf u_\mu:=\mathrm S(s) \boldsymbol\mu$ and note that $\mathrm V(s)\boldsymbol\mu=\gamma\mathbf u_\mu$. Taking $\mathbf v=\mathbf u_\mu$ as test function in Propositions \ref{prop:4.3} and \ref{prop:4.8}, it follows that
\begin{equation}\label{eq:5.4}
\langle \boldsymbol\lambda,\mathrm V(s)\boldsymbol\mu\rangle_\Gamma=\langle\boldsymbol\lambda,\gamma\mathbf u_\mu\rangle_\Gamma = a(\mathbf u_\lambda,\mathbf u_\mu)+ s\,(\mathbf u_\lambda,\mathbf u_\mu)_{\mathbb R^d},
\end{equation}
which proves (a).

Let now $\mathbf u_\lambda:=\mathrm S(s)\boldsymbol\lambda$.
By (a) and \eqref{eq:5.4}, we can write
\[
\overline s^{1/2} \langle\overline{\boldsymbol\lambda},\mathrm V(s)\boldsymbol\lambda\rangle_\Gamma 
=
 \overline s^{1/2} \langle \boldsymbol\lambda,\mathrm V(s)\overline{\boldsymbol\lambda}\rangle_\Gamma
=
 \overline s^{1/2} a(\mathbf u_\lambda,\overline{\mathbf u_\lambda})+s^{1/2} |s| (\mathbf u_\lambda,\overline{\mathbf u_\lambda})_{\mathbb R^d},
\]
which proves (b).

If $\mathbf v\in \widehat{\mathbf V}(\mathbb R^d)$, then
\[
\langle\mathbf n,\gamma\mathbf v\rangle_\Gamma=\int_\Gamma\gamma\mathbf v\cdot\mathbf n=\int_{\Omega_-}\mathrm{div}\,\mathbf v=0
\]
and therefore $\mathrm S(s)\mathbf n=\mathbf 0$. If $\mathbf V(s)\boldsymbol\lambda=\mathbf 0$, then, by (b) it follows that $\mathbf u_\lambda=\mathbf 0$. Therefore, the associated pressure $p_\lambda=\mathrm S_p\boldsymbol\lambda$ satisfies
$\nabla p_\lambda=\mathbf 0$ in $\mathbb R^d\setminus\Gamma$,
and is decaying at infinity. This proves that $p_\lambda\in \mathrm{span}\,\{\chi_{\Omega_-}\}$ and therefore $\boldsymbol\lambda=\jump{\mathbf t(\mathbf u_\lambda,p_\lambda)}=\jump{\mathbf t(\mathbf 0,p_\lambda)}\in \mathrm{span}\,\{\mathbf n\}$. This finishes the proof of (c).

Because of Lemma \ref{lemma:4.1} there exists a bounded operator
\begin{equation}\label{eq:5.5}
\gamma^\dagger:\mathbf H^{1/2}_n(\Gamma)\to \widehat{\mathbf V}(\mathbb R^d) \qquad \gamma\gamma^\dagger\boldsymbol\phi=\boldsymbol\phi \qquad \forall\boldsymbol\phi\in \mathbf H^{1/2}_n(\Gamma).
\end{equation}
For $\boldsymbol\lambda\in \mathbf H^{-1/2}_m(\Gamma)$ we define $\mathbf u_\lambda=\mathrm S(s)\boldsymbol\lambda$. Then
\begin{alignat*}{4}
|\langle \boldsymbol\lambda,\boldsymbol\phi\rangle_\Gamma | & =  |\langle \boldsymbol\lambda,\gamma \gamma^\dagger\boldsymbol\phi\rangle_\Gamma| 
& & \mbox{by \eqref{eq:5.5}}\\
&=|a(\mathbf u_\lambda,\gamma^\dagger\boldsymbol\phi)+s(\mathbf u_\lambda,\gamma^\dagger\boldsymbol\phi)| &\qquad & \mbox{by Propositions \ref{prop:4.3} \& \ref{prop:4.8}}\\
&\le \triple{\mathbf u_\lambda}_{(s)}\triple{\gamma^\dagger\boldsymbol\phi}_{(s)} & & \mbox{by \eqref{eq:5.3}}\\
& \le C_\Gamma \alpha_2(s) \triple{\mathbf u_\lambda}_{(s)}\|\boldsymbol\phi\|_{1/2,\Gamma} & & \mbox{by \eqref{eq:5.1} and the trace theorem.}
\end{alignat*}
Using Lemma \ref{lemma:4.2} it follows that
\[
\|\boldsymbol\lambda\|_{-1/2,\Gamma}\le C \alpha_2(s) \triple{\mathbf u_\lambda}_{(s)}.
\]
Therefore by (b),
\begin{equation}\label{eq:5.20}
\mathrm{Re}\,\langle \overline s^{1/2} \overline{\boldsymbol\lambda},\mathrm V(s)\boldsymbol\lambda\rangle_\Gamma = \omega \triple{\mathbf u_\lambda}^2_{(s)} \ge C\, \frac{\omega}{\alpha_2(s)^2} \|\boldsymbol\lambda\|_{-1/2,\Gamma}^2 \qquad \forall \boldsymbol\lambda \in \mathbf H^{-1/2}_m(\Gamma).
\end{equation}
The remainder of the proof is straightforward, using \eqref{eq:5.2} to get the final lower bound.
\end{proof}

\paragraph{Remark.} As part of the proof of Proposition \ref{prop:N3.3}(b) we have shown that
\begin{equation}\label{eq:5.a}
\| \boldsymbol\lambda\|_{-1/2,\Gamma} \le C \frac{|s|^{1/2}}{\underline\omega}\triple{\mathrm S(s)\boldsymbol\lambda}_{(s)} \qquad \forall\boldsymbol\lambda\in \mathbf H^{-1/2}_m(\Gamma) \quad \forall s\in \mathbb C_\star.
\end{equation}

\begin{proposition}[Bounds for the single layer potential]\label{prop:5.5}
There exists $C>0$ such that
\[
\alpha_1(s)\triple{\mathrm S(s)\boldsymbol\lambda}_{(1)}\le 
\triple{\mathrm S(s)\boldsymbol\lambda}_{(s)} \le C \frac{\alpha_2(s)}{\omega} \|\boldsymbol\lambda\|_{-1/2,\Gamma} \qquad \forall \boldsymbol\lambda \in \mathbf H^{-1/2}(\Gamma).
\]
\end{proposition}

\begin{proof}
Let $\mathbf u_\lambda:=\mathrm S(s)\boldsymbol\lambda$. Then:
\begin{alignat*}{4}
\omega \triple{\mathbf u_\lambda}_{(s)}^2 &=\mathrm{Re}\, \langle \overline s^{1/2}\overline{\boldsymbol\lambda},\mathrm V(s)\boldsymbol\lambda\rangle_\Gamma & \qquad & \mbox{by Proposition \ref{prop:N3.3}(b)}\\
& \le  |s|^{1/2} |\langle \overline{\boldsymbol\lambda},\gamma\mathbf u_\lambda\rangle_\Gamma| & & \mbox{by definition of $\mathrm V(s)$}\\
& \le C_\Gamma |s|^{1/2} \|\boldsymbol\lambda\|_{-1/2,\Gamma} \triple{\mathbf u_\lambda}_{(1)} & & \mbox{by the trace theorem}\\
& \le C_\Gamma \frac{|s|^{1/2}}{\alpha_1(s)}\| \boldsymbol\lambda\|_{-1/2,\Gamma} \triple{\mathbf u_\lambda}_{(s)}  & & \mbox{by \eqref{eq:5.1}}\\
&= C_\Gamma \alpha_2(s) \| \boldsymbol\lambda\|_{-1/2,\Gamma} \triple{\mathbf u_\lambda}_{(s)}.
\end{alignat*}
\end{proof}

\begin{proposition}[Dirichlet solver]\label{prop:5.6}
Let $\boldsymbol\phi\in \mathbf H^{1/2}_n(\Gamma)$ and
$
\mathbf u:=\mathrm S(s)\mathrm V(s)^{-1}\boldsymbol\phi.
$
Then
\[
\triple{\mathbf u}_{(1)}\le C \,\frac{\alpha_2(s)}{\alpha_1(s)}\frac{|s|^{1/2}}\omega\,\|\boldsymbol\phi\|_{1/2,\Gamma}=C\,\frac{\max\{1,|s|\}}\omega \,\|\boldsymbol\phi\|_{1/2,\Gamma}\le C \, \frac{|s|}{\underline\omega^2\,\omega}\,\|\boldsymbol\phi\|_{1/2,\Gamma}.
\]
\end{proposition}

\begin{proof} Note first that $\gamma\mathbf u=\boldsymbol\phi$.
By Propositions \ref{prop:4.3} and \ref{prop:4.8},
\[
a(\mathbf u,\mathbf v)+s(\mathbf u,\mathbf v)_{\mathbb R^d}=0 \qquad \forall\mathbf v\in \widehat{\mathbf V}(\mathbb R^d)\quad\mbox{such that}\quad \gamma\mathbf v=\mathbf 0.
\]
Therefore, using the lifting operator in \eqref{eq:5.5} and defining $\mathbf u_0:=\mathbf u-\gamma^\dagger\boldsymbol\phi$, it follows that
\[
a(\mathbf u_0,\overline{\mathbf u_0})+s(\mathbf u_0,\overline{\mathbf u_0})_{\mathbb R^d} = -a(\gamma^\dagger\boldsymbol\phi,\overline{\mathbf u_0})-s(\gamma^\dagger\boldsymbol\phi,\overline{\mathbf u_0})_{\mathbb R^d} 
\]
and thus
\[
\omega \triple{\mathbf u_0}_{(s)}^2 = \mathrm{Re}\, \Big( \overline s^{1/2} a(\mathbf u_0,\overline{\mathbf u_0})+s^{1/2}|s| (\mathbf u_0,\overline{\mathbf u_0})_{\mathbb R^d}\Big)\le  |s|^{1/2} \triple{\mathbf u_0}_{(s)}\triple{\gamma^\dagger\boldsymbol\phi}_{(s)}
\]
by \eqref{eq:5.3}.
This implies that
\[
\triple{\mathbf u}_{(s)} \le \Big(1+\frac{|s|^{1/2}}\omega\Big) \triple{\gamma^\dagger\boldsymbol\phi}_{(s)} \le C \alpha_2(s) \frac{|s|^{1/2}}\omega \|\boldsymbol\phi\|_{1/2,\Gamma}.
\]
The remainder of the proof is straightforward.
\end{proof}

\paragraph{Summary of bounds.}
In terms of $\omega$ (see \eqref{eq:5.0}), $\alpha_1(s)$ and $\alpha_2(s)$ (see \eqref{eq:5.2}), we can write
\begin{alignat*}{4}
\|\mathrm S(s)\| & \le C \, \frac{\alpha_2(s)}{\omega\,\alpha_1(s)} = C\,\frac1\omega \max\{|s|^{1/2},|s|^{-1/2}\} & \qquad & \mbox{by Proposition \ref{prop:5.5}}\\
\|\mathrm V(s)\| & \le C \, \frac{\alpha_2(s)}{\omega\,\alpha_1(s)} = C\,\frac1\omega \max\{|s|^{1/2},|s|^{-1/2}\} & \qquad & \mbox{since $\mathrm V(s)=\gamma\mathrm S(s)$}\\
\|\mathrm V(s)^{-1}\| & \le C \frac{|s|^{1/2}\max\{1,|s|\}}\omega & & \mbox{by Proposition \ref{prop:N3.3}(d)}\\
\|\mathrm S(s)\mathrm V(s)^{-1}\| & \le C \frac{\alpha_2(s)}{\alpha_1(s)}\, \frac{|s|^{1/2}}\omega = C\,\frac{\max\{1,|s|\}}\omega & & \mbox{by Proposition \ref{prop:5.6}}
\end{alignat*}
The operator norms above are the natural ones using the spaces $\mathbf H^1(\mathbb R^d)$, $\mathbf H^{1/2}(\Gamma)$, and $\mathbf H^{-1/2}(\Gamma)$, where it corresponds.
Using \eqref{eq:5.2}, to bound
\[
\frac{\alpha_2(s)}{\alpha_1(s)}\le \frac{|s|^{1/2}}{\underline\omega^2} \qquad \max\{1,|s|\}\le \frac{|s|}{\underline\omega^2} \qquad \underline\omega:=\min\{1,\omega\}, 
\]
we can obtain a new set of bounds, valid for all $s\in \mathbb C_\star$:
\begin{subequations}\label{eq:5.6}
\begin{alignat}{4}
\|\mathrm S(s)\| +\|\mathrm V(s)\| & \le C \, \frac{|s|^{1/2}}{\omega\underline\omega^2} \\
\|\mathrm V(s)^{-1}\| & \le C\, \frac{|s|^{3/2}}{\omega\underline\omega^2} \\
\|\mathrm S(s)\mathrm V(s)^{-1}\| & \le C \, \frac{|s|}{\omega\underline\omega^2}
\end{alignat}
\end{subequations}

\section{An evolutionary exterior Stokes problem}\label{sec:4}

Given a Banach space $X$, we consider the set of causal $\mathcal C^k$ $X$-valued functions
\[
\mathcal C^k_+(\mathbb R;X):=\{ \phi  : \mathbb R \to X\,:\, \phi\in \mathcal C^k(\mathbb R;X),\quad \phi(t)=0 \,\, \forall t\le 0\}.
\]
The space of bounded linear operators from $X$ to $Y$ (for Hilbert spaces $X$ and $Y$) will be denoted $\mathcal B(X,Y)$.)

\paragraph{Abstract setting.}
The starting point is an operator valued holomorphic function $\mathrm F:\mathbb C_\star \to \mathcal B(X,Y)$, such that
\begin{equation}\label{eq:6.1}
\|\mathrm F(s)\|_{X\to Y} \le C_{\mathrm F}(\mathrm{Re}\,s^{1/2}) | s |^\mu \qquad \forall s\in \mathbb C_\star, \qquad 0\le \mu < 1,
\end{equation}
where
\begin{equation}\label{eq:6.2}
C_{\mathrm F}:(0,\infty)\to (0,\infty)\mbox{ is non-increasing, and } 
C_{\mathrm F}(\omega) \le C \omega^{-\ell} \quad \omega \to 0, \, \ell>0.
\end{equation}
In particular, there exists an $\mathcal B (X,Y)$-valued casual distribution $f$, whose Laplace transform is $\mathrm F(s)$. The following result is based on \cite[Lemma 2.2]{LuSc:1992}. Its proof is given in an appendix. We note that in comparison with \cite{LuSc:1992} we are more demanding in terms of regularity of $g$, but we pay attention to behavior of constants as $t$ grows. 

\begin{proposition}\label{prop:6.1}
Let $f$ be such that its Laplace transform $\mathrm F:\mathbb C_\star\to \mathcal B(X,Y)$ satisfies \eqref{eq:6.1}-\eqref{eq:6.2}, and let $g\in \mathcal C^1_+(\mathbb R;X)$. Then $f*g\in \mathcal C_+(\mathbb R;Y)$ and
\[
\|(f*g)(t)\|_Y \le C_\mu \min\{1,t^{\ell/2+1-\mu}\} \max_{0\le \tau\le t} \| g'(\tau)\|_X \qquad \forall t\ge 0.
\]
\end{proposition}

\begin{corollary}\label{cor:6.2}
Let $f$ be such that its Laplace transform $\mathrm F:\mathbb C_\star\to \mathcal B(X,Y)$ satisfies
\[
\|\mathrm F(s)\|_{X\to Y}\le C_{\mathrm F}(\mathrm{Re}\,s^{1/2}) |s|^{1+\mu} \qquad s \in \mathbb C_\star, \quad 0\le \mu <1,
\]
where $C_{\mathrm F}$ satisfies \eqref{eq:6.2}. Then, for all $g\in \mathcal C^2_+(\mathbb R;X)$, we have that $f*g\in \mathcal C_+(\mathbb R;Y)$ and 
\[
\|(f*g)(t)\|_Y \le C_\mu \min\{1,t^{\ell/2+1-\mu}\} \max_{0\le \tau\le t} \| g''(\tau)\|_X \qquad \forall t\ge 0.
\]
\end{corollary}

\begin{proof}
Let $\partial_t^{-1} f$ be the distribution whose transform is $s^{-1}\mathrm F(s)$. Then $f*g=\partial_t^{-1} f * g'$ and we can apply Proposition \ref{prop:6.1} to $\partial_t^{-1} f $ and $g'$.
\end{proof}

\subsection{Estimates for the single layer potential and operator}

Because of the bounds \eqref{eq:5.6} and the Payley-Wiener theorem, there exists a causal distribution $\mathcal S$ with values in $\mathcal B(\mathbf H^{-1/2}(\Gamma),\mathbf H^1(\mathbb R^d))$, whose Laplace transform is $\mathrm S$. The convolution operator
$
\mathcal S * \boldsymbol\lambda
$
for any causal $\mathbf H^{-1/2}(\Gamma)$-valued distribution is the {\bf single layer potential} for the Stokes operator in the time domain. The distribution $\mathcal V:=\gamma\mathcal S$ (with Laplace transform $\mathrm V$) gives rise to the convolution operator
$
\mathcal V*\boldsymbol\lambda = \gamma (\mathcal S*\boldsymbol\lambda) = (\gamma\mathcal S)*\boldsymbol\lambda,
$
known as the {\bf single layer operator} for the Stokes problem in the time domain.

\begin{proposition}
Let $\boldsymbol\lambda \in \mathcal C^1_+(\mathbb R;\mathbf H^{-1/2}(\Gamma))$. Then $\mathcal S*\boldsymbol\lambda$ and $\mathcal V*\boldsymbol\lambda$ are continuous functions and
\begin{eqnarray*}
\| (\mathcal S*\boldsymbol\lambda)(t)\|_{1,\mathbb R^d} & \le & C \min\{ 1,t^2\} \max_{0\le \tau\le t} \| \boldsymbol\lambda'(\tau)\|_{-1/2,\Gamma} \qquad \forall t \ge 0,\\
\| (\mathcal V*\boldsymbol\lambda)(t)\|_{1/2,\Gamma} & \le & C \min\{ 1,t^2\} \max_{0\le \tau\le t} \| \boldsymbol\lambda'(\tau)\|_{-1/2,\Gamma} \qquad \forall t \ge 0.
\end{eqnarray*}
\end{proposition}

\begin{proof}
The distribution $\mathcal S$ satisfies the hypotheses of Proposition \ref{prop:6.1} (i.e. \eqref{eq:6.1}-\eqref{eq:6.2}) with $\mu=1/2$ and $\ell=3$. The result is then a direct consequence of Proposition \ref{prop:6.1}.
\end{proof}

\begin{proposition}\label{prop:6.4}
Let $\boldsymbol\phi\in \mathcal C^2_+(\mathbb R;\mathbf H^{1/2}_n(\Gamma))$. Then there exists a unique causal distribution $\boldsymbol\lambda$ with values in $\mathbf H^{-1/2}_m(\Gamma)$ such that
$
\mathcal V*\boldsymbol\lambda=\boldsymbol\phi.
$
Moreover $\boldsymbol\lambda \in \mathcal C_+(\mathbb R;\mathbf H^{-1/2}(\Gamma)$ and the associated potential
$
\mathbf u=\mathcal S*\boldsymbol\lambda
$
is also continuous as a function of $t$. Finally, we have the bounds:
\begin{eqnarray*}
\| \boldsymbol\lambda(t)\|_{-1/2,\Gamma} & \le & C \min\{ 1,t^{2}\} \max_{0\le \tau\le t} \| \boldsymbol\phi''(\tau)\|_{1/2,\Gamma} \qquad \forall t \ge 0,\\
\| \mathbf u(t)\|_{1,\mathbb R^d} & \le & C \min\{ 1,t^{5/2}\} \max_{0\le \tau\le t} \| \boldsymbol\phi''(\tau)\|_{-1/2,\Gamma} \qquad \forall t \ge 0.
\end{eqnarray*}
\end{proposition}

\begin{proof}
There is a slightly delicate argument to show uniqueness. By causality, we can look at the equation assuming that $\boldsymbol\phi$ is compactly supported. This means that $\boldsymbol\phi$ has a Laplace transform $\boldsymbol\Phi$, and using the Payley-Wiener theorem and the bounds \eqref{eq:5.6}, there is a unique solution whose Laplace transform is $\mathrm V(s)^{-1}\boldsymbol\Phi(s)$.

Once existence and uniqueness is settled, the bounds follow from Corollary \ref{cor:6.2}. For $\mathcal V^{-1}$ we have the hypotheses of Corollary \ref{cor:6.2} with $\mu=1/2$ and $\ell=3$. For $\mathcal S*\mathcal V^{-1}$, we have the hypotheses with $\mu=0$ and $\ell=3$.
\end{proof}

\subsection{The exterior Dirichlet problem}

Our starting point is the velocity field on $\Gamma$ at all times
$
\boldsymbol\phi\in \mathcal C^2_+(\mathbb R;\mathbf H^{1/2}_n(\Gamma)).
$
Using the result of Proposition \ref{prop:6.4}, we produce
\begin{equation}\label{eq:6.7}
\boldsymbol\lambda \in \mathcal C_+(\mathbb R;\mathbf H^{-1/2}_m(\Gamma)) \qquad \mathbf u \in \mathcal C_+(\mathbb R;\widehat{\mathbf V}(\mathbb R^d))
\end{equation}
satisfying
\begin{equation}\label{eq:6.8}
\mathcal V*\boldsymbol\lambda=\boldsymbol\phi \qquad \mbox{and}\qquad \mathbf u=\mathcal S*\boldsymbol\lambda.
\end{equation}
We finally construct the pressure field, by applying the (time-independent) pressure part of the single layer operator for the steady-state Stokes equation:
\begin{equation}\label{eq:6.9}
p(t):=\mathrm S_p \, (\boldsymbol\lambda(t))=\mathrm S_p\boldsymbol\lambda(t).
\end{equation}
By Propositions \ref{prop:2.1} and \ref{prop:2.2}, it follows that
\begin{equation}\label{eq:6.10}
p\in \mathcal C_+(\mathbb R;L^2(\mathbb R^3)) \qquad\mbox{and}\qquad p\in \mathcal C_+(\mathbb R;L^2(B)) \quad \mbox{for any bounded set $B\subset \mathbb R^2$}.
\end{equation}
In the remainder of this section, it is necessary to clarify that all differential operators in the space variables will be used in the sense of distributions in $\mathbb R^d\setminus\Gamma$. If $\partial_{x_i}$ is the differentiation operator with respect to the $i$-th variable, it is well known that $\partial_{x_i}:L^2(\mathbb R^d\setminus\Gamma) \to H^{-1}(\mathbb R^d\setminus\Gamma)$ is bounded.

\begin{proposition}
Let $\mathbf u$ and $p$ be given by \eqref{eq:6.7}-\eqref{eq:6.9}. Then
\begin{subequations}\label{eq:6.11}
\begin{alignat}{4}\label{eq:6.11a}
\dot{\mathbf u}(t)-\nu \Delta \mathbf u(t)+\nabla p(t) &=\mathbf 0 &\qquad &\forall t\ge 0,\\
\label{eq:6.11b}
\mathrm{div}\, \mathbf u(t) &=0 & & \forall t\ge 0,\\
\label{eq:6.11c}
\gamma \mathbf u(t) &=\boldsymbol\phi(t) & & \forall t\ge 0,\\
\mathbf u(0) &=\mathbf 0.
\label{eq:6.11d}
\end{alignat}
\end{subequations}
For any $t\ge 0$, the equation \eqref{eq:6.11a} is to be understood in the sense of distributions in $\mathbb R^d\setminus\Gamma$. Finally,
\begin{equation}\label{eq:6.12}
\mathbf u \in \mathcal C^1_+(\mathbb R;\mathbf H^{-1}(\mathbb R^d\setminus\Gamma)).
\end{equation}
\end{proposition}

\begin{proof}
Note first that \eqref{eq:6.11b} is satisfied because $\mathbf u$ is a continuous function with values in the space of solenoidal fields $\widehat{\mathbf V}(\mathbb R^d)$. The initial condition \eqref{eq:6.11d} is a consequence of the fact that $\mathbf u$ is continuous and causal.

By causality, we can assume that $\boldsymbol\phi^{(k)}(t)$ is bounded in $t$ for $k\le 2$ (this does not affect the generality of the result), and therefore, the Laplace transforms of $\boldsymbol\phi$, $\mathbf u$ and $p$ exist for $s\in \mathbb C_\star$. Moreover, they satisfy
\[
\mathbf U(s) =\mathrm S(s)\mathbf V(s)^{-1}\boldsymbol\Phi(s), \qquad P(s)=\mathrm S_p\mathbf V(s)^{-1}\boldsymbol\Phi(s)
\]
and therefore
\begin{equation}\label{eq:6.13}
s\mathbf U(s)-\nu \Delta \mathbf U(s)+\nabla P(s)=\mathbf 0 \qquad \forall s\in \mathbb C_\star
\end{equation}
and
\begin{equation}\label{eq:6.14}
\gamma \mathbf U(s)=\boldsymbol\Phi(s) \qquad \forall s\in \mathbb C_\star.
\end{equation}
The equality \eqref{eq:6.14} proves the boundary condition \eqref{eq:6.11c} in the time domain.

Note now that $\Delta \mathbf u\in \mathcal C_+(\mathbb R;\mathbf H^{-1}(\mathbb R^d\setminus\Gamma))$. In the three dimensional case, it is clear from \eqref{eq:6.10} that $\nabla p\in \mathcal C_+(\mathbb R;\mathbf H^{-1}(\mathbb R^3\setminus\Gamma))$. In the two dimensional case, we have to use the decomposition of Proposition \ref{prop:2.2} and the fact that $\nabla(\mathbf p_\infty\cdot\mathbf e_\ell)\in \mathbf L^2(\mathbb R^2)$ in order to prove that $\nabla p\in \mathcal C_+(\mathbb R;\mathbf H^{-1}(\mathbb R^2\setminus\Gamma))$. In third place $\mathbf u\in \mathcal C_+(\mathbb R;\mathbf H^{-1}(\mathbb R^d\setminus\Gamma)$ and therefore $\dot{\mathbf u}$ is a causal distribution with values in $\mathbf H^{-1}(\mathbb R^d\setminus\Gamma)$. Taking Laplace transforms of $\Delta \mathbf u$, $\dot{\mathbf u}$ and $\nabla p$ and using \eqref{eq:6.13} we show that
\begin{equation}\label{eq:6.15}
\dot{\mathbf u}=\nu \Delta \mathbf u -\nabla p.
\end{equation}
This equation is to be understood in the sense of causal distributions with values in $\mathbf H^{-1}(\mathbb R^3\setminus\Gamma)$. However, as we have seen above, the right hand side of \eqref{eq:6.15} is a continuous causal $\mathbf H^{-1}(\mathbb R^d\setminus\Gamma)$-valued function. This proves \eqref{eq:6.12} and the equality \eqref{eq:6.15} is satisfied pointwise in time, that is, we have proved \eqref{eq:6.11a} as equality of elements of $\mathbf H^{-1}(\mathbb R^d\setminus\Gamma)$ for all $t$. In its turn, this can be understood as a distributional equation in $\mathbb R^d\setminus\Gamma$ for all $t$.
\end{proof}

\section{Galerkin semidiscretization in space}\label{sec:5}

Let $\mathbf X_h \subset \mathbf H^{-1/2}_m(\Gamma)$ be a finite dimensional space. The semidiscretized BIE for the exterior Dirichlet problem starts with causal Dirichlet data $\boldsymbol\phi: \mathbb R \to \mathbf H^{1/2}_n(\Gamma)$, looks for a causal function
$
\boldsymbol\lambda^h :\mathbb R \to \mathbf X_h 
$
(i.e. $\boldsymbol\lambda^h(t)=0$ for all $t\le 0$)
such that
\begin{equation}\label{eq:7.1}
\langle \boldsymbol\mu^h,(\mathcal V*\boldsymbol\lambda^h)(t)\rangle_\Gamma =\langle \boldsymbol\mu^h,\boldsymbol\phi(t)\rangle_\Gamma \quad \forall \boldsymbol\mu^h \in \mathbf X_h, \quad \forall t,
\end{equation}
and finally constructs
\begin{equation}\label{eq:7.1b}
\mathbf u^h:=\mathcal S* \boldsymbol\lambda^h, \qquad p^h=\mathrm S_p\boldsymbol\lambda^h.
\end{equation}
The semi-discretized integral equation \eqref{eq:7.1} can be also written in the following abstract form
$
(\mathcal V*\boldsymbol\lambda^h)(t)-\boldsymbol\phi(t)\in \mathbf X_h^\circ,
$
for all $t$.

\subsection{The Galerkin solver}

We first study properties of the Galerkin solver, i.e., the operator $\mathrm G_h(s):\mathbf H^{1/2}_n(\Gamma)\to \mathbf X_h$ defined by $\boldsymbol\lambda^h:=\mathrm G_h(s)\boldsymbol\phi$, where
\begin{equation}\label{eq:7.1c}
\boldsymbol\lambda^h \in \mathbf X_h \qquad \mbox{s.t.}\qquad \langle\boldsymbol\mu^h,\mathrm V(s)\boldsymbol\lambda^h\rangle_\Gamma =\langle\boldsymbol\mu^h,\boldsymbol\phi\rangle_\Gamma \quad \forall \boldsymbol\mu^h \in \mathbf X_h.
\end{equation}
We will also be interested in the associated velocity field
$
\mathbf u^h=\mathrm S(s)\boldsymbol\lambda^h=\mathrm S(s)\mathrm G_h(s)\boldsymbol\phi.
$
Note that the space $\mathbf X_h$ was chosen to work in the time domain and can be taken to be real-valued. In the context of Laplace transforms, it has to be closed by conjugation, which is equivalent to taking the same space with complex scalars to create linear combinations.

\begin{proposition}[Bound for the Galerkin solver]\label{prop:7.1}
There exists a constant independent of $h$ such that
\[
\|\mathrm G_h(s)\|\le C \frac{|s|^{3/2}}{\omega\underline\omega^2} \qquad \mbox{and}\qquad \| \mathrm S(s)\mathrm G_h(s)\|\le C \frac{|s|}{\omega\underline\omega^2}.
\]
\end{proposition}

\begin{proof}
The first estimate is a direct consequence of the coercivity estimate of Proposition \ref{prop:N3.3}(d). To show the second one, we need to replicate the proof of Proposition \ref{prop:5.6}. Set $\boldsymbol\lambda^h=\mathrm G_h(s)\boldsymbol\phi$. By Propositions \ref{prop:4.3} and \ref{prop:4.8} (solenoidal variational form for the single layer potential)
\[
\left[
\begin{array}{l}
\mathbf u^h \in \widehat{\mathbf V}(\mathbb R^d),\\[1.5ex]
a(\mathbf u^h,\mathbf v)+s\,(\mathbf u^h,\mathbf v)_{\mathbb R^d}=\langle\boldsymbol\lambda^h,\gamma\mathbf v\rangle_\Gamma \qquad \forall \mathbf v\in \widehat{\mathbf V}(\mathbb R^d).
\end{array}
\right.
\]
Consider now the closed space
\begin{equation}\label{eq:7.2}
\widehat{\mathbf V}_h(\mathbb R^d):=\{\mathbf v\in \widehat{\mathbf V}(\mathbb R^d)\,:\, \gamma \mathbf v\in \mathbf X_h^\circ\}= \{\mathbf v\in \widehat{\mathbf V}(\mathbb R^d)\,:\, \langle\boldsymbol\mu_h,\gamma \mathbf v\rangle_\Gamma=0 \quad\forall\boldsymbol\mu^h \in \mathbf X_h\}.
\end{equation}
Then $\mathbf u^h$ is the unique solution of the problem
\begin{equation}\label{eq:7.3}
\left[
\begin{array}{l}
\mathbf u^h \in \widehat{\mathbf V}(\mathbb R^d),\qquad
\gamma\mathbf u^h-\boldsymbol\phi\in \mathbf X_h^\circ,\\[1.5ex]
a(\mathbf u^h,\mathbf v)+s\,(\mathbf u^h,\mathbf v)_{\mathbb R^d}=0 \qquad \forall \mathbf v\in \widehat{\mathbf V}_h(\mathbb R^d).
\end{array}
\right.
\end{equation}
Using the lifting \eqref{eq:5.5}, we can proceed as in the proof of Proposition \ref{prop:5.6} and decompose $\mathbf u^h=\mathbf w^h+\gamma^\dagger\boldsymbol\phi$, where $\mathbf w^h\in \widehat{\mathbf V}_h(\mathbb R^d)$ and finally show that
\[
\triple{\mathbf u^h}_{(s)} \le C \frac{\max\{1,|s|\}}{\omega}\|\boldsymbol\phi\|_{1/2,\Gamma},
\]
from where the bound follows.
\end{proof}

\subsection{The Galerkin error operator}

The Galerkin projector looks at the discrete problem from the point of view of the exact solution. We can define it as the operator $\mathrm G_h(s)\mathrm V(s):\mathbf H^{-1/2}_m(\Gamma) \to \mathbf X_h$, or equivalently by setting $\boldsymbol\lambda^h$ as the solution of the discrete equations
\begin{equation}\label{eq:7.A}
\boldsymbol\lambda^h \in \mathbf X_h \qquad \mbox{s.t.}\qquad \langle\boldsymbol\mu^h,\mathrm V(s)\boldsymbol\lambda^h\rangle_\Gamma =\langle\boldsymbol\mu^h,\mathrm V(s)\boldsymbol\lambda\rangle_\Gamma \quad \forall \boldsymbol\mu^h \in \mathbf X_h.
\end{equation}
Instead of studying this projection we will study the complementary projection, that corresponds to the error of the Galerkin semidiscretization. We thus consider the operators
$
\mathrm E_h(s):=\mathrm G_h(s)\mathrm V(s)-\mathrm I
$
and
$
\mathrm S(s)\mathrm E_h(s).
$
Note that while $\mathrm G_h(s)\mathrm V(s)$ is a projection onto $\mathbf X_h$ for all $s\in \mathbb C_\star$, the range of $\mathrm E_h(s)$ varies with $s$.

\begin{proposition}[Bounds for the Galerkin error operator]\label{prop:7.2}
There exists a constant independent of $h$ such that
\[
\|\mathrm E_h(s)\| \le C \frac{|s|}{\omega\underline\omega^2} \qquad\mbox{and}\qquad \|\mathrm S(s)\mathrm E_h(s)\| \le C\,\frac{|s|^{1/2}}{\omega\underline\omega^2}
\]
\end{proposition}

\begin{proof}
Let $\boldsymbol\lambda\in \mathbf H^{-1/2}_m(\Gamma)$ and
$
\mathbf w^h:=\mathrm S(s) \mathrm E_h(s)\boldsymbol\lambda=\mathrm S(s)(\boldsymbol\lambda^h-\boldsymbol\lambda),
$
where $\boldsymbol\lambda^h$ is the solution of \eqref{eq:7.A}.
Proceeding as in the proof of Proposition \ref{prop:5.5} we can show that
\begin{alignat*}{4}
\omega \triple{\mathbf w^h}_{(s)}^2 &= \mathrm{Re}\,\langle\overline s^{1/2} (\overline{\boldsymbol\lambda}^h-\overline{\boldsymbol\lambda}),\mathrm V(s)(\boldsymbol\lambda^h-\boldsymbol\lambda)\rangle_\Gamma &\qquad &\mbox{by Proposition \ref{prop:N3.3}(b)}\\
&= -\mathrm{Re}\,\langle\overline s^{1/2} \overline{\boldsymbol\lambda},\mathrm V(s)(\boldsymbol\lambda^h-\boldsymbol\lambda)\rangle_\Gamma &\qquad &\mbox{by Galerkin orthogonality}\\
& \le C_\Gamma |s|^{1/2} \|\boldsymbol\lambda\|_{-1/2,\Gamma} \triple{\mathbf w^h}_{(1)} & & \mbox{by the trace theorem}.
\end{alignat*}
Therefore, by \eqref{eq:5.1} it follows that 
\[
 \qquad \triple{\mathbf w^h}_{(s)}\le C \frac{|s|^{1/2}}{\omega\,\underline\omega}\|\boldsymbol\lambda\|_{-1/2,\Gamma} \qquad \mbox{and}\qquad \triple{\mathbf w^h}_{(1)}\le C\frac{|s|^{1/2}}{\omega\underline\omega^2}\|\boldsymbol\lambda\|_{-1/2,\Gamma}.
\]
The second bound gives the estimate for $\|\mathrm S(s)\mathrm E_h(s)\|$.
Using now \eqref{eq:5.a}, we can bound
\[
\| \mathrm E_h(s)\boldsymbol\lambda\|_{-1/2,\Gamma} =\|\boldsymbol\lambda^h-\boldsymbol\lambda\|_{-1/2,\Gamma} \le C \frac{|s|^{1/2}}{\underline\omega} \triple{\mathbf w^h}_{(s)}\le  C \frac{|s|}{\omega\underline\omega^2}\|\boldsymbol\lambda\|_{-1/2,\Gamma},
\]
which finishes the proof.
\end{proof}

Since $\mathrm I-\mathrm E_h(s)+\mathrm I=\mathrm G_h(s)\mathrm V(s)$ is a projection onto $\mathbf X_h$, if $\boldsymbol\Pi_h:\mathbf H^{-1/2}_m(\Gamma) \to \mathbf X_h$ is any projection onto the discrete space we can write
\begin{equation}\label{eq:7.5}
\mathrm E_h(s)=\mathrm E_h(s)(\mathrm I-\boldsymbol\Pi_h).
\end{equation}
This decomposition will be used to derive error estimates.

\subsection{Bounds in the time domain}

\begin{proposition}[Stability; bounds with respect to data]
Let $\boldsymbol\phi \in \mathcal C^2_+(\mathbb R;\mathbf H^{1/2}_n(\Gamma))$ and let $\boldsymbol\lambda^h$ be the solution of \eqref{eq:7.1} and $\mathbf u^h$ be given by \eqref{eq:7.1b}. Then
\[
\boldsymbol\lambda^h \in \mathcal C_+(\mathbb R;\mathbf X_h) \qquad \mathbf u^h \in \mathcal C_+(\mathbb R;\mathbf H^1(\mathbb R^d))
\]
and
\begin{eqnarray*}
\|\boldsymbol\lambda^h(t)\|_{-1/2,\Gamma} &\le & C \max\{ 1,t^2\}\max_{0\le \tau\le t} \|\boldsymbol\phi''(\tau)\|_{1/2,\Gamma},\\
\|\mathbf u^h(t)\|_{1,\mathbb R^d} &\le & C \max\{ 1,t^{5/2}\}\max_{0\le \tau\le t} \|\boldsymbol\phi''(\tau)\|_{1/2,\Gamma}.
\end{eqnarray*}
\end{proposition}

\begin{proof} It is a direct consequence of Corollary \ref{cor:6.2} (abstract result in the time domain) and Proposition \ref{prop:7.1} (Laplace domain bounds for the Galerkin solver).
\end{proof}

For the error estimates we use the orthogonal projection operator $\boldsymbol\Pi_h:\mathbf H^{-1/2}_m(\Gamma) \to\mathbf X_h$. We first give an estimate of the velocity field, which requires much less regularity in time. We will next give an estimate for the density, which will in  turn give an estimate for the pressure field.

\begin{proposition}[Error estimate for the velocity field]
Assume that $\boldsymbol\lambda \in \mathcal C^1_+(\mathbb R;\mathbf H^{-1/2}_m(\Gamma))$. Then
\[
\|\mathbf u^h(t)-\mathbf u(t)\|_{1,\mathbb R^d}\le C \max\{1,t^2\} \max_{0\le \tau\le t}\| \boldsymbol\lambda'(\tau)-\boldsymbol\Pi_h\boldsymbol\lambda'(\tau)\|_{-1/2,\Gamma}.
\]
\end{proposition}

\begin{proof}
This is the time-domain version of the second bound of Proposition \ref{prop:7.2}, using Proposition \ref{prop:6.1} (with $\mu=1/2$ and $\ell=2$) and the identity \eqref{eq:7.5} in order to introduce the orthogonal projector.
\end{proof}

\begin{proposition}[Error estimate for density and pressure field]
Assume that $\boldsymbol\lambda \in \mathcal C^2_+(\mathbb R;\mathbf H^{-1/2}_m(\Gamma))$. Then
\begin{eqnarray*}
\|\boldsymbol\lambda^h(t)-\boldsymbol\lambda(t)\|_{-1/2,\Gamma} &\le& C \max\{1,t^{5/2}\} \max_{0\le \tau\le t}\| \boldsymbol\lambda''(\tau)-\boldsymbol\Pi_h\boldsymbol\lambda''(\tau)\|_{-1/2,\Gamma}\\
\|p^h(t)-p(t)\|_{B} &\le& C \max\{1,t^{5/2}\} \max_{0\le \tau\le t}\| \boldsymbol\lambda''(\tau)-\boldsymbol\Pi_h\boldsymbol\lambda''(\tau)\|_{-1/2,\Gamma},
\end{eqnarray*}
where $B=\mathbb R^3$ or $B$ is any bounded open set in $\mathbb R^2$.
\end{proposition}

\begin{proof}
Apply Corollary \ref{cor:6.2} to the conclusions of Proposition \ref{prop:7.2} and use the identity \eqref{eq:7.5}. For the bound on the pressure, use Propositions \ref{prop:2.1} and \ref{prop:2.2}.
\end{proof}

\section{Full discretization and numerical experiments}\label{sec:6}

We finally do a full discretization of equations \eqref{eq:7.1} and \eqref{eq:7.1b} using Lubich's multistep-based Convolution Quadrature \cite{Lubich:1988}. We next give a short introduction to this black-box technology applied to our particular problem. More implementation details can be found in \cite{BaSc:2012} and \cite{HaSa:2014} (although for wave propagation problems). Before we introduce the method, let us also mention that there is a faster version (which changes the implementation, but not the method itself) called the fast and oblivious CQ method \cite{ScLoLu:2006} that we will not deal with in this paper.

Let us choose a basis $\{\bs\mu_j\,:\,j=1,\ldots,N\}$ for $\mathbf X_h$, a time-step $\kappa>0$, and let us consider the uniform grid in time $t_n:=n\kappa$, for $n\ge 0$. The data are sampled in time and tested to define vectors
\[
\bs\phi_n\in \mathbb R^N, \qquad \phi_{n,j}:=\langle\bs\mu_j,\bs\phi(t_n)\rangle_\Gamma.
\]	
The transfer operator corresponding to the convolution with $\mathcal V$ is defined as a matrix-valued function of $s\in \mathbb C_\star$:
\[
\mathbf V(s)\in \mathbb C^{N\times N}, \qquad \mathbf V_{ij}(s)=\langle\bs\mu_i,\mathrm V(s)\bs\mu_j\rangle_\Gamma.
\]
The CQ discretization of \eqref{eq:7.1} starts with a Taylor expansion
\[
\mathbf V\left(\smallfrac1\kappa\delta(\zeta)\right)=\sum_{n=0}^\infty \mathbf V_n(\kappa) \zeta^n, \qquad
\mbox{where}\qquad \delta(\zeta):=\sum_{\ell=1}^p \frac1\ell(1-\zeta)^\ell.
\]
The function $\delta(\zeta)$ is the characteristic function of the BDF method of order $p$. Lubich's theoretical results hold for $p\le 6$. Note that most of the matrices $\mathbf V_n(\kappa)\in \mathbb R^{N\times N}$ do not have to be computed in the practical implementation of the method. The discretization of \eqref{eq:7.1} looks for the sequence of vectors $\bs\lambda_n \in \mathbb R^N$ given by the recurrence:
\begin{equation}\label{eq:66.1}
\mathbf V_0(\kappa)\bs\lambda_n =\bs\phi_n-\sum_{m=1}^n \mathbf V_m(\kappa)\bs\lambda_{n-m}, \qquad n\ge 0.
\end{equation}
If $\bs\lambda_n=(\lambda_{n,1},\ldots,\lambda_{n,N})$, we then reconstruct the discrete function $\bs\lambda^h_n:=\sum_{j=1}^N \lambda_{j,n} \bs\mu_j\in \mathbf X_h$. The discrete densities provide the discrete pressure field
\begin{equation}\label{eq:66.2}
p^h_n:=\mathrm S_p\bs\lambda^h_n.
\end{equation}
To compute the discrete velocity field we use another postprocessing of the discrete densities
\begin{equation}\label{eq:66.3}
\mathbf u^h_n:=\sum_{m=0}^n \mathrm S_m(\kappa)\bs\lambda_{n-m},
\qquad
\mbox{where}
\qquad
\mathrm S(\smallfrac1\kappa\delta(\zeta))=\sum_{n=0}^\infty \mathrm S_n(\kappa)\zeta^n.
\end{equation}
The convergence result follows from \cite[Theorem 5.1]{Lubich:1988} by using Proposition \ref{prop:7.1}.

\begin{proposition}
Let $p$ be the order of the BDF method used for the CQ discretization. Assume that $\bs\phi\in \mathcal C^{p+1}_+(\mathbb R;\mathbf H^{1/2}(\Gamma))\cap \mathcal C^{p+2}([0,\infty);\mathbf H^{1/2}(\Gamma))$. Then
\begin{eqnarray}
\| \mathbf u^h(t_n)-\mathbf u^h_n\|_{1,\mathbb R^d}
	&\le & C_1 \kappa^p \max_{0\le \tau\le t} \|\bs\phi^{(p+2)}(\tau)\|_{1/2,\Gamma},\\
\| p^h(t_n)-p^h_n\|_B
	&\le & C_2 \kappa^p \max_{0\le \tau\le t} \|\bs\phi^{(p+2)}(\tau)\|_{1/2,\Gamma},
\end{eqnarray}
where $B=\mathbb R^3$ or $B$ is any bounded open set in $\mathbb R^2$. The constants $C_1$ and $C_2$ depend on $t$, and $C_2$ depends on $B$ in the two-dimensional case. For small $t$, $C_1\le C t$ and $C_2\le C t^{1/2}$.
\end{proposition}

\paragraph{A first numerical experiment.} In order to be able to compare our method with an exact solution we will solve problem \eqref{eq:6.11} in the domain $\Omega_-=(-1,1)^2$. We choose the data so that the exact solution is
\[
\mathbf u(t)(x,y)=\sin^9(t) H(t) \left[\begin{array}{c} 2x \\ -2y \end{array}\right],
\qquad
p(t)(x,y)=-9\sin^8(t)\cos(t)H(t) (x^2-y^2),
\]
where $H$ is the Heaviside function. The exact density $\bs\lambda(t)$ is not known. Note that even if the exact solution is smooth, there is no guarantee that $\bs\lambda(t)$ will be a smooth function in the space variable. We integrate from $t=0$ to $t=1$. For discretization in space we choose a uniform partition of $\Gamma=\partial\Omega_-$ in $N$ equally sized elements $\{e_1,\ldots,e_N\}$ where $N$ is a multiple of four. We then consider the spaces
\[
\mathbf X_h^+:=\{ \bs\lambda^h:\Gamma \to \mathbb R^2\,:\, \bs\lambda^h|_{e_j} \in \mathcal P^1(e_j)^2\quad \forall j\},
\qquad
\mathbf X_h:=\mathbf X_h^+\cap \mathbf H^{-1/2}_m(\Gamma), 
\]
where $\mathcal P^1$ is the space of polynomials of degree less than or equal to one. Instead of building a basis for $\mathbf X_h$, we will enforce densities to be in $\mathbf X_h$ using two Lagrange multipliers. This only affects the matrix $\mathbf V_0(\kappa)$ in \eqref{eq:66.1}. Time discretization is carried out with CQ using BDF(3) as ODE solver in the background, using $M$ time-steps to reach $t=1$. We then compute errors for the pressure and the velocity
\[
\max_j |\mathbf u(1)(x_j,y_j)-\mathbf u^h_M(x_j,y_j)| \qquad
\max_j |p(1)(x_j,y_j)-p^h_M(x_j,y_j)|,
\]
where
\[
(x_1,y_1):=(-0.5,-0.5), \quad 
(x_2,y_2):=(0.3,0.7),\quad
(x_3,y_3):=(0.6,0.2).
\]
If $\bs\lambda$ were smooth as a function of the space variable (which we do not know), the expected convergence order predicted by the theory would be $h^{2.5}+\kappa^3$, where $h=1/N$ and $\kappa=1/M$. This does not take into account the possible regularization effects of the potentials. We note that, to the best of our knowledge, there is no theory of time-domain integral equations that is able to predict higher order convergence in weaker norms. The results are shown in Table \ref{table:1}

\begin{table}[h]
\begin{center}
\begin{tabular}{c|c|c|c|c|c}
$N$	&	$M$ 	&	\text{errU}		& 	\text{e.c.r. }	& 	\text{errP}		& 	\text{e.c.r} 	\\
\hline
4 	& 	10	& 	1.6448e-02	& 		-		&	6.9116e-02	&	-			\\
\hline
8	&	20 	&	9.5414e-03	& 	0.79 	& 	6.3904e-02	& 	0.11	\\
\hline
16 	& 	40 	&	1.2200e-03	& 	2.97	& 	2.4554e-03	& 	4.70	\\
\hline
32 	& 	80	& 	5.8683e-05	& 	4.38	&	8.4062e-04	&	1.55	\\
\hline
64 	&    160 	&  	1.7639e-05	& 	1.73	& 	1.3247e-04	&	2.67	\\
\hline
128 	&     320	&	2.2716e-06	&	2.96	& 	1.0263e-05	&	3.69	\\
\hline
256 	&     640 	&	1.9787e-07	&	3.52	& 	2.9564e-07	& 	5.12	\\
\end{tabular}
\end{center}
\caption{Results at time $t=1$ measured on three points interior to a square. Time-stepping is carried out with BDF(3)-based CQ and discontinuous piecewise linear functions are used for space discretization.}\label{table:1}
\end{table}

\paragraph{A second experiment.} We deal with the same exact solution but now use the unit circle as the domain.  We measure the same errors, based now on three observation points placed at $(0,0)$, $(1/2,1/2)$ and $(-.6,.1)$. We used BDF(3) as the time stepping method, taking $M$ time steps to reach $t=1$. For space discretization we use piecewise constant functions on a uniform grid (in parameter space), with $N$ elements, and reduced integration. The fully discrete method that we obtain is equivalent to a Nystr\"om method of the class given in \cite{DoLuSa:2014}. Because of the smoothness of the domain, the density is a smooth function of the space variables and it is to be expected that order three convergence can be observed for potential postprocessings, although this has never been proved for problems in the time domain. The results are reported in Table \ref{table:2}.
 
\begin{table}[h]
\begin{center}
\begin{tabular}{c|c|c|c|c|c}
$N$    &    $M$     &    \text{errU}        & \text{e.c.r. }    &     \text{errP}        & \text{e.c.r}     \\
\hline
20     &     20    &     1.2285e-03    &         -     &    3.9793e-03    &    -            \\
\hline
40    &    40     &    1.3750e-04    &     3.16     &     4.0498e-04    &    3.30    \\
\hline
80    &     80     &    1.7287e-05    &     2.99     &     4.9458e-05    &    3.04    \\
\hline
160     &     160    &     2.1636e-06    &     2.99         &    6.1078e-06    &    3.02    \\
\hline
320     &     320     &      2.7053e-07    &     3.00         &     7.5887e-07    &    3.01    \\
\hline
640     &     640    &    3.3819e-08    &    3.00     &     9.4578e-08    &    3.00    \\
\end{tabular}
\end{center}
\caption{Results at time $t=1$ measured on three points interior to the unit circle. Time-stepping is carried out with BDF(3)-based CQ. Piecewise constant functions with reduced integration are used in the space variable.}\label{table:2}
\end{table}

\paragraph{An illustration.} We finally show some snapshots of a time simulation for an exterior problem. The Dirichlet data is of the form $\bs\phi(\mathbf x,t)=f(t) (1/\sqrt2,1/\sqrt2)$, where $f$ is a smooth causal function whos shape can be seen in the third column of Figure \ref{fig:2}. The domain is a smooth six sided start. In figure \ref{fig:2} we show vorticity and pressure at different times.

\newpage

\begin{figure}[H]
\begin{center}

\includegraphics[scale=.25]{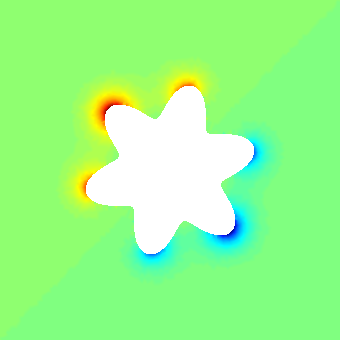}\hspace{3ex}
\includegraphics[scale=.25]{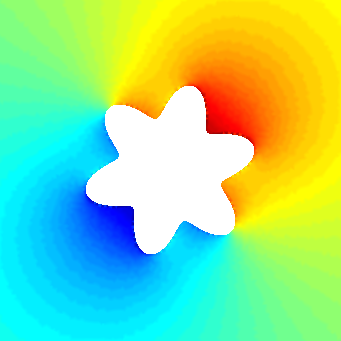}\hspace{3ex}
\includegraphics[scale=.25]{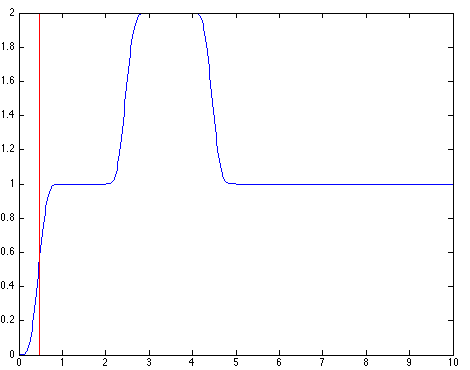}

\includegraphics[scale=.25]{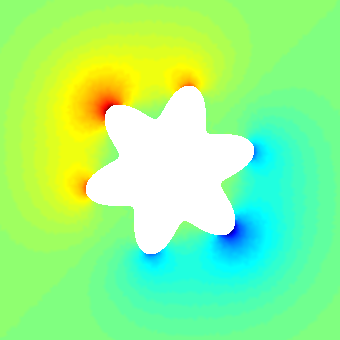}\hspace{3ex}
\includegraphics[scale=.25]{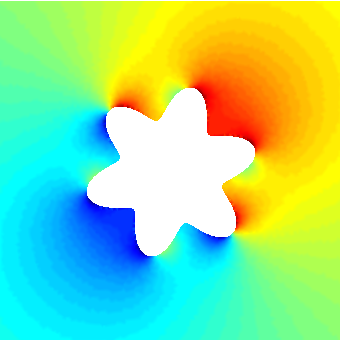}\hspace{3ex}
\includegraphics[scale=.25]{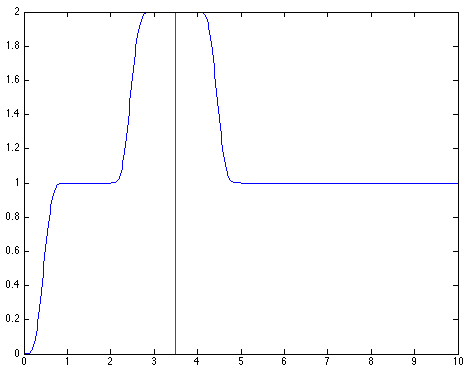}

\includegraphics[scale=.25]{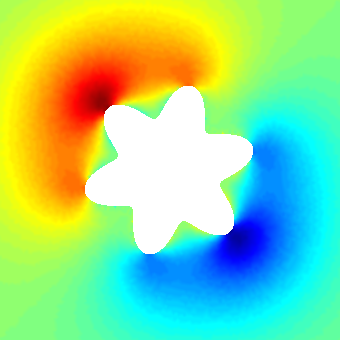}\hspace{3ex}
\includegraphics[scale=.25]{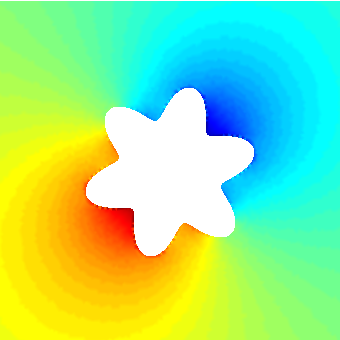}\hspace{3ex}
\includegraphics[scale=.25]{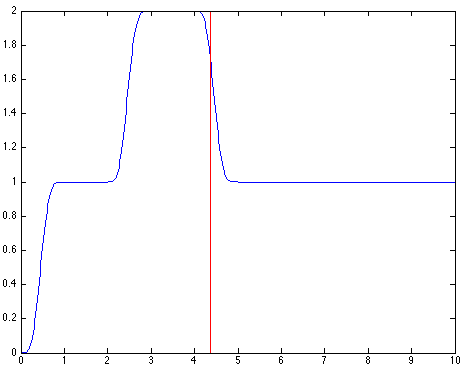}

\includegraphics[scale=.25]{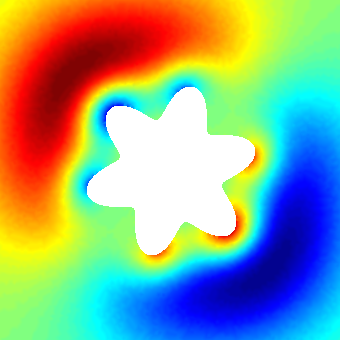}\hspace{3ex}
\includegraphics[scale=.25]{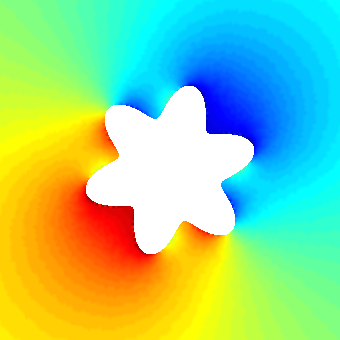}\hspace{3ex}
\includegraphics[scale=.25]{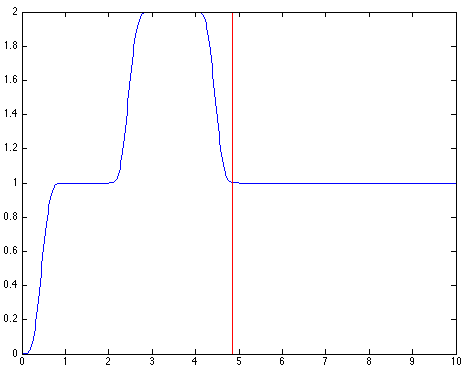}

\includegraphics[scale=.25]{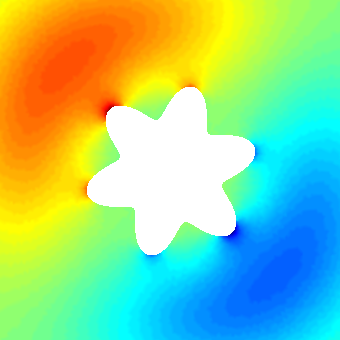}\hspace{3ex}
\includegraphics[scale=.25]{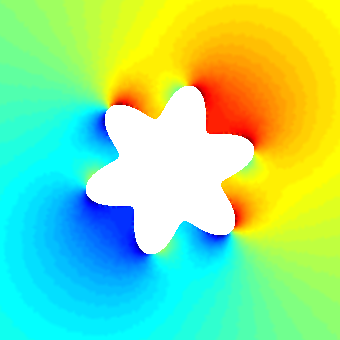}\hspace{3ex}
\includegraphics[scale=.25]{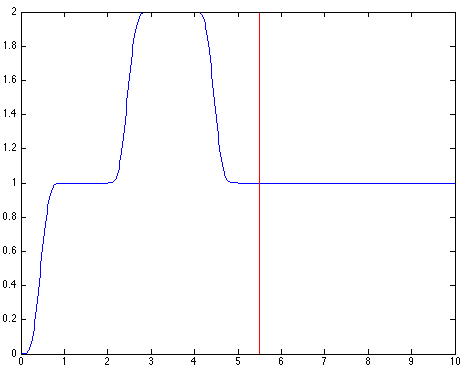}

\includegraphics[scale=.25]{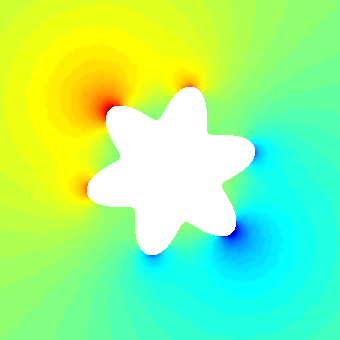}\hspace{3ex}
\includegraphics[scale=.25]{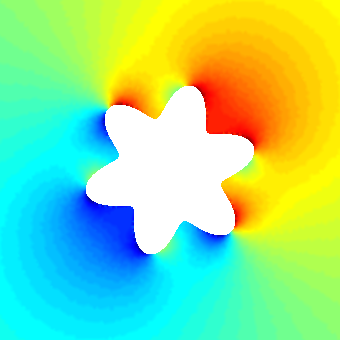}\hspace{3ex}
\includegraphics[scale=.25]{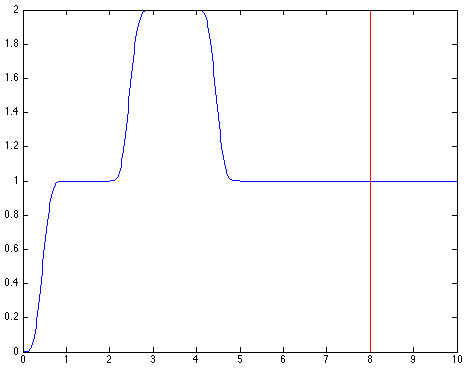}

\end{center}
\caption{Six stages of the Stokes flow produced by a fixed direction non-uniform flow. The right column shows the time as a verticle bar running along the graph of the time-variable function that marks the speed of the flow. The left column shows vorticity and the middle column, pressure.}\label{fig:2}
\end{figure}

\bibliography{referencesStokesBEM}

\appendix

\section{Proof of Proposition \ref{prop:6.1}}\label{app:A}

Since the result gives estimates of the convolution $f*g$, when $g\in\mathcal C^1_+(\mathbb R;X)$, and the convolution with $f$ is a causal operator, we can assume (without loss of generality) that $g$ and $g'$ are uniformly bounded. The following function
\[
a(s,t):=\frac{\mathrm d}{\mathrm dt}\int_0^t e^{s(t-\tau)} g(\tau)\mathrm d\tau=\int_0^t e^{s(t-\tau)}g'(\tau)\mathrm d\tau
\]
is well defined for all $t\in [0,\infty)$ and $s\in \mathbb C$. It is then possible to show (see \cite[Lemma 2.2]{LuSc:1992}) that
\begin{equation}\label{eq:6.3}
(f*g)(t)=\frac1{2\pi \imath} \int_\Gamma s^{-1} \mathrm F(s) a(s,t)\mathrm d s
\end{equation}
for a variety of integration contours. (This is shown by proving that the Laplace transform of the function in the right-hand side of  \eqref{eq:6.3} is $\mathrm F\,\mathrm G$.) Here we choose a two-parameter family of contours (see Figure \ref{fig:1}), formed by three pieces:
\begin{eqnarray*}
(-\infty,-c] \ni \rho & \longmapsto & z_-(\rho):=-\rho \,e^{-\imath(\pi-\phi)},\\
{}[-(\pi-\phi),\pi-\phi]\ni \rho &\longmapsto & z_0(\rho):=c e^{\imath\rho},\\
{}[c,\infty) \ni \rho &\longmapsto & z_+(\rho) := \rho \,e^{\imath (\pi-\phi)}.
\end{eqnarray*}
The parameter $c>0$ will play a decisive role in the estimates below, while $\phi\in (0,\pi/2)$ does not seem to be relevant for the following bounds.
\begin{figure}[H]
\begin{center}
\includegraphics[width=8cm]{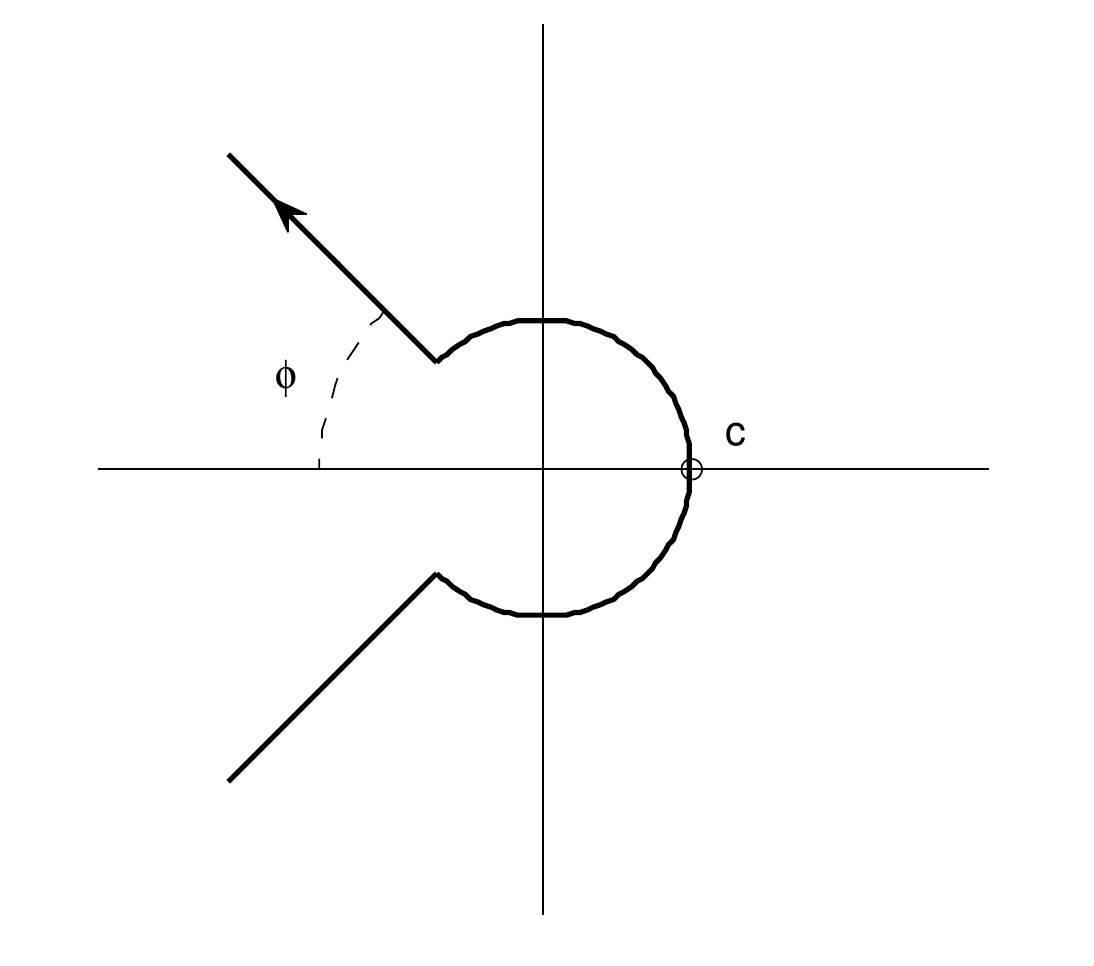}
\end{center}\caption{The contours in the proof of Proposition \ref{prop:6.1}}\label{fig:1}
\end{figure}
We first note that for all $t\ge 0$ and $s\in \mathbb C$,
\begin{equation}\label{eq:6.4}
\| a(s,t)\|\le \| g'\|_t 
\left\{ 
\begin{array}{ll} 
t\, e^{t\,\mathrm{Re}\,s} & \mathrm{Re}\,s\ge 0,\\
t, & \mathrm{Re}\,s\le 0,\\
\frac1{|\mathrm{Re}\,s|}, &\mathrm{Re}\,s<0,
\end{array}
\right.\qquad\mbox{where}\qquad  \| g'\|_t:= \max_{0\le \tau\le t} \| g'(\tau)\|.
\end{equation}
We start by bounding the part of the contour integral \eqref{eq:6.3} that arises from the central path $\Gamma_0=\{ z_0(\rho):|\rho|\le \pi-\phi\}$. Using
\[
|z_0(\rho)|= |z_0'(\rho)|=c, \quad \mathrm{Re}\, z_0(\rho)^{1/2}=\sqrt{c}\cos\smallfrac\theta2\ge \sqrt{c}\cos\smallfrac{\pi-\phi}2=\sqrt{c}\sin\smallfrac\phi2, \quad \mathrm{Re}\,z_0(\rho)\le c,	
\]
and \eqref{eq:6.4}, we can bound
\[
\| s^{-1} \mathrm F(s) \| \le C_{\mathrm F}(\sqrt c\,\sin\smallfrac\phi2)\, c^{\mu-1} \qquad \| a(s,t)\|\le t e^{c\,t} \| g'\|_t, \qquad s\in \Gamma_0
\]
and therefore
\begin{equation}\label{eq:6.5}
\Big\| \int_{\Gamma_0} s^{-1}\mathrm F(s) a(s,t)\mathrm d s\Big\|\le 2(\pi-\phi) C_{\mathrm F}(\sqrt c \sin\smallfrac\phi2)\,c^\mu\,e^{ct}\,t\,\|g'\|_t.
\end{equation}
In $\Gamma_+:=\{ z_+ (\rho): \rho\ge c\}$, we have
\[
|z_+(\rho)|=\rho, \quad |z_+'(\rho)|=1, \quad \mathrm{Re}\,z_+(\rho)^{1/2} =\sqrt\rho \sin\smallfrac\phi2 \ge \sqrt{c}\sin\smallfrac\phi2,\quad |\mathrm{Re}\,z_+(\rho)|=\rho \cos\phi,
\]
and therefore (the bound in $\Gamma_-$ can be done simultaneously)
\begin{eqnarray}\nonumber
\Big\| \int_{\Gamma_\pm} s^{-1}\mathrm F(s) a(s,t)\mathrm d s\Big\| &\le& C_{\mathrm F}(\sqrt c\sin\smallfrac\phi2) \| g'\|_t \frac1{\cos\phi} \int_c^\infty \theta^{\mu-2}\mathrm d \theta\\
&=& C_{\mathrm F}(\sqrt c\sin\smallfrac\phi2) \| g'\|_t \frac1{\cos\phi}\,\frac{c^{\mu-1}}{1-\mu}.\label{eq:6.6}
\end{eqnarray}
When $t\le 1$, we can take $c=1$ in \eqref{eq:6.5} and \eqref{eq:6.6} to bound
\[
\| (f*g)(t)\|\le 2 \Big(  (\pi-\phi) t+\frac1{(1-\mu)\cos\phi}\Big)\,C_{\mathrm F}(\sin\smallfrac\phi2)\,\|g'\|_t\qquad t \le 1.
\]
When $t\ge 1$, we take $c=1/t$ and obtain
\[
\| (f*g)\| \le 2 \Big(  (\pi-\phi) t+\frac1{(1-\mu)\cos\phi}\Big)t^{1-\mu}\,C_{\mathrm F}(t^{-1/2}\sin\smallfrac\phi2)\,\|g'\|_t\qquad t \ge 1.
\]
Using \eqref{eq:6.2} bound of the statement is established. Continuity of $f*g$ follows from the representation \eqref{eq:6.3} and the Dominated Convergence Theorem.

\section{An equivalent integral equation}\label{sec:AppB}

We start by describing the formulation (at the continuous and semidiscrete level) for the Brinkman equation. The aim of this formulation is to incorportate the restrictions for test and trial functions to be in $\mathbf H^{-1/2}_m(\Gamma)$ as part of the integral operator. In order to do this, we define the operator
\[
\widetilde{\mathrm V}(s):=\mathrm V(s)+\langle\punto,\mathbf m\rangle_\Gamma \mathbf m: \mathbf H^{-1/2}(\Gamma) \to \mathbf H^{1/2}(\Gamma). 
\]
This is the operator associated to the bilinear form
$
\langle\boldsymbol\mu,\mathrm V(s)\boldsymbol\lambda\rangle_\Gamma + \langle\boldsymbol\mu,\mathbf m\rangle_\Gamma \,\langle\boldsymbol\lambda,\mathbf m\rangle_\Gamma.
$

\begin{proposition}\label{prop:8.1}
Let $\boldsymbol\phi\in \mathbf H^{1/2}_n(\Gamma)$. Then
\[
\left.
\begin{array}{r}
	\mathrm V(s)\boldsymbol\lambda =\boldsymbol\phi \\
	\langle\boldsymbol\lambda,\mathbf m\rangle_\Gamma=0
\end{array}\right\} \quad \Longleftrightarrow \quad \widetilde{\mathrm V}(s)\boldsymbol\lambda=\boldsymbol\phi.
\]
Moreover $\widetilde{\mathrm V}(s): \mathbf H^{-1/2}(\Gamma) \to \mathbf H^{1/2}(\Gamma)$ is invertible for all $s\in \mathbb C_\star$ and
\[
\|\widetilde{\mathrm V}(s)^{-1}\|\le C\frac{|s|^{3/2}}{\omega\underline\omega^2}.
\]
\end{proposition}

\begin{proof}
The first assertion is straightforward, given the fact that $\mathrm V(s)\boldsymbol\lambda\in \mathbf H^{1/2}_n(\Gamma)$ for all $\boldsymbol\lambda$. To prove invertibility we derive a coercivity estimate. The decomposition of Lemma \ref{lemma:4.2} can be done in the following way
\[
\boldsymbol\lambda=\boldsymbol\lambda_0+ c(\boldsymbol\lambda)\,\mathbf n \qquad c(\boldsymbol\lambda):=\frac{\langle\boldsymbol\lambda,\mathbf m\rangle_\Gamma}{\langle \mathbf n,\mathbf m\rangle_\Gamma}, \qquad \boldsymbol\lambda_0 \in \mathbf H^{-1/2}_m(\Gamma),
\]
and $|c(\boldsymbol\lambda)|+\|\boldsymbol\lambda_0\|_{-1/2,\Gamma}\equiv \|\boldsymbol\lambda\|_{-1/2,\Gamma}.$ It is then easy to show that
\[
\langle\overline{\boldsymbol\lambda},\mathrm V(s)\boldsymbol\lambda\rangle_\Gamma =\langle \overline{\boldsymbol\lambda_0},\mathrm V(s)\boldsymbol\lambda_0\rangle_\Gamma + |c(\boldsymbol\lambda)|^2 \langle\mathbf n,\mathbf m\rangle_\Gamma^2.
\]
By \eqref{eq:5.20}
\[
|s|^{1/2}|\langle\overline{\boldsymbol\lambda},\mathrm V(s)\boldsymbol\lambda\rangle_\Gamma| \ge \mathrm{Re}\,\langle \overline s^{1/2}\overline{\boldsymbol\lambda},\mathrm V(s)\boldsymbol\lambda \ge C \frac{\omega}{\alpha_2(s)^2}\|\boldsymbol\lambda_0\|_{-1/2,\Gamma}^2+ C \omega |c(\boldsymbol\lambda)|^2 
\]
and therefore (using that $\underline\omega\le |s|^{1/2}$ and the bounds \eqref{eq:5.2}),
\[
|\langle\overline{\boldsymbol\lambda},\mathrm V(s)\boldsymbol\lambda\rangle_\Gamma| \ge C\,\frac{\omega\underline\omega^2}{|s|^{3/2}}\Big( \|\boldsymbol\lambda_0\|_{-1/2,\Gamma}^2 + |c(\boldsymbol\lambda)|^2\Big),
\]
which finishes the proof.
\end{proof} 

For semidiscretization in space, we choose a finite dimensional space $\mathbf X_h^+ \subset \mathbf H^{-1/2}(\Gamma)$ such that $\mathbf n \in \mathbf X_h^+$. (In the case of polyhedral boundaries, this is easily verified if piecewise constant functions are elements of the space.) If we define the space $\mathbf X_h:=\mathbf X_h^+ \cap\mathbf H^{-1/2}_m(\Gamma)$, we have a stable decomposition $\mathbf X_h^+=\mathbf X_h\oplus \mathrm{span}\,\{\mathbf n\}$. The semidiscrete equations in the Laplace domain \eqref{eq:7.1c} are equivalent to
\[
\boldsymbol\lambda_h \in \mathbf X_h^+ \qquad \mbox{s.t.}\qquad \langle\boldsymbol\mu_h,\widetilde{\mathrm V}(s)\boldsymbol\lambda_h\rangle_\Gamma =\langle\boldsymbol\mu_h,\boldsymbol\phi\rangle_\Gamma \quad \forall \boldsymbol\mu_h \in \mathbf X_h^+.
\]
In the time domain, they correspond to looking for a causal function
$
\boldsymbol\lambda_h :\mathbb R \to \mathbf X_h^+ 
$
such that
\[
\langle \boldsymbol\mu_h,(\mathcal V*\lambda_h)(t)\rangle_\Gamma +\langle\boldsymbol\mu_h,\mathbf m\rangle_\Gamma
\langle\boldsymbol\lambda_h(t),\mathbf m\rangle_\Gamma=\langle \boldsymbol\mu_h,\boldsymbol\phi(t)\rangle_\Gamma \quad \forall \boldsymbol\mu_h \in \mathbf X_h^+, \quad \forall t.
\]
Because of Proposition \ref{prop:8.1}, all the preceding bounds for the semidiscrete case can be easily translated to this new formulation.

\end{document}